\documentclass[1 [leqno,11pt]{amsart}
\usepackage{amssymb, amsmath}
\def\refer#1{~\ref{#1}}
\def\refeq#1{~(\ref{#1})}
\def\ccite#1{~\cite{#1}}

\def\longformule#1#2{
\displaylines{ \qquad{#1} \hfill\cr \hfill {#2} \qquad\cr } }
\def\inte#1{
\displaystyle\mathop{#1\kern0pt}^\circ }

\let\pa=\partial
\let\al=\alpha
\let\b=\beta

\let\d=\delta

\let\r=\rho
\let\s=\sigma
\let\f=\frac

\let\p=\psi
\let\om=\omega

\let\D=\Delta

\let\Om=\Omega
\let\wt=\widetilde
\let\wh=\widehat

\def\cB{{\mathcal B}}
\def\cC{{\mathcal C}}

\def\cE{{\mathcal E}}
\def\cF{{\mathcal F}}

\def\cH{{\mathcal H}}

\def\cL{{\mathcal L}}

\def\cS{{\mathcal S}}

\def\cV{{\mathcal V}}

\def\grad{\nabla}

\def\dH{\dot{H}}
\def\dB{\dot{B}}

\def\virgp{\raise 2pt\hbox{,}}
\def\cdotpv{\raise 2pt\hbox{;}}

\def\eqdefa{\buildrel\hbox{\footnotesize def}\over =}
\def\Id{\mathop{\rm Id}\nolimits}

\def\C{\mathop{\mathbb C\kern 0pt}\nolimits}
\def\DD{\mathop{\mathbb D\kern 0pt}\nolimits}
\def\EE{\mathop{  {\mathbb E \kern 0pt}}\nolimits}
\def\K{\mathop{\mathbb K\kern 0pt}\nolimits}
\def\N{\mathop{\mathbb N\kern 0pt}\nolimits}
\def\Q{\mathop{\mathbb Q\kern 0pt}\nolimits}
\def\R{\mathop{\mathbb R\kern 0pt}\nolimits}
\def\SS{\mathop{\mathbb S\kern 0pt}\nolimits}
\def\ZZ{\mathop{\mathbb Z\kern 0pt}\nolimits}
\def\TT{\mathop{\mathbb T\kern 0pt}\nolimits}
\def\P{\mathop{\mathbb P\kern 0pt}\nolimits}
\newcommand{\ds}{\displaystyle}

\newcommand{\Z}{{\ZZ}}

\def\dv{\mbox{\rm div}}

\def\dive{\mathop{\rm div}\nolimits}

\def\qq{pour tout\ }

\def\Supp{\mathop{\rm Supp}\nolimits\ }

\def\no{\noindent}
\def\na{\nabla}
\def\p{\partial}
\def\th{\theta}

\def\vcurl{v^{\rm h}_{\rm curl}}
\def\nablah{\nabla_{\rm h}}
\def\vdiv{v^{\rm h}_{\rm div}}

\newcommand{\beq}{\begin{equation}}
\newcommand{\eeq}{\end{equation}}
\newcommand{\ben}{\begin{eqnarray}}
\newcommand{\een}{\end{eqnarray}}
\newcommand{\beno}{\begin{eqnarray*}}
\newcommand{\eeno}{\end{eqnarray*}}
\newcommand{\andf}{\quad\hbox{and}\quad}
\newcommand{\with}{\quad\hbox{with}\quad}
\newtheorem{defi}{Definition}[section]
\newtheorem{thm}{Theorem}[section]
\newtheorem{lem}{Lemma}[section]

\newtheorem{prop}{Proposition}[section]

\begin{document}
\title[Regularity criterion for 3-D
 Navier-Stokes system]
{On the critical one component regularity for 3-D
 Navier-Stokes system: general case}
 \author[J.-Y. CHEMIN]{Jean-Yves Chemin}
\address [J.-Y. Chemin]%
{Laboratoire J.-L. Lions, UMR 7598 \\
Universit\'e Pierre et Marie Curie, 75230 Paris Cedex 05, FRANCE }
\email{chemin@ann.jussieu.fr}
\author[P. ZHANG]{Ping Zhang}%
\address[P. Zhang]
 {Academy of
Mathematics $\&$ Systems Science and  Hua Loo-Keng Key Laboratory of
Mathematics, The Chinese Academy of Sciences\\
Beijing 100190, CHINA } \email{zp@amss.ac.cn}
\author[Z. ZHANG]{Zhifei Zhang}\address[Z. ZHANG]
{School of  Mathematical Science, Peking University, Beijing 100871,
P. R. CHINA} \email{zfzhang@math.pku.edu.cn}

\date{2/22/2015}

\begin{abstract} Let us consider an initial data~$v_0$  for the homogeneous incompressible
 3D Navier-Stokes equation with vorticity  belonging to~$L^{\frac 32}\cap L^2$.
  We prove that if the solution associated with~$v_0$ blows up at a finite time~$T^\star$,
  then for any~$p$ in~$]4,\infty[$, and any unit vector~$e$ of~$\R^3$, the~$L^p$ norm in time with
  value in~$\dot{H}^{\frac 12+\frac  2 p }$ of~$(v|e)_{\R^3}$ blows up at~$T^\star$.
\end{abstract}
\maketitle

Littlewood-Paley Theory\

\setcounter{equation}{0}
\section{Introduction}
In this work,  we investigate  necessary  conditions for  breakdown of regularity of regular solutions to
the following 3-D homogeneous incompressible Navier-Stokes system
\begin{equation*}
(NS)\qquad \left\{\begin{array}{l}
\displaystyle \pa_t v + \dv (v\otimes v) -\D v+\grad \Pi=0, \qquad (t,x)\in\R^+\times\R^3, \\
\displaystyle \dv\, v = 0, \\
\displaystyle  v|_{t=0}=v_0,
\end{array}\right. \label{1.1}
\end{equation*}
where $v=(v^1,v^2, v^3)$ stands for the   velocity of the fluid and
$\Pi$ for the pressure. We shall study   necessary conditions for
blowing up in the framework of Fujita and Kato solutions.
 Let us sum up the fact about this theory introduced in\ccite{fujitakato} by H. Fujita and T. Kato that will be relevant in our work.
\begin{thm}
\label{fujitakato+}
{\sl Let~$v_0$ be in the homogenneous Sobolev
space~$\dot H^{\frac 12}$. There exists a unique  maximal
solution~$v$ in the space~$ C([0,T^\ast[;\dH ^{\frac 1 2   })\cap
L^2_{\rm loc}([0,T^\star[;\dH ^{\frac 32})$. If~$T^\star$ is finite,
then  we have, for any~$p$ in~$[2,\infty[$
\beq
\label{blowupbasic}
\int_0^{T^\star}\|v(t,\cdot)\|_{\dH^{\frac 1 2   +\frac 2 p}}^p\,dt=\infty.
\eeq }
\end{thm}
The limiting  case when~$p= \infty$ namely that fact that if there is blow up in finite time~$T^\star$, then~$\ds \limsup_{t\rightarrow T^\star}
\|v(t)\|_{\dot H^{\frac 12}}$ is infinite is a consequence of the work\ccite{ISS}  of L. Escauriaza, G. Seregin and V. \u{S}ver\'{a}k.

In all that follows, we consider initial data~$v_0$ with vorticity
~$\Omega_0\eqdefa \na\times v_0$  belonging to~$ L^{\frac 32}.$ Let
us mention that dual Sobolev embedding implies that~$L^{\frac 32}
\hookrightarrow  \dot H^{-\frac 12}$ which together with Biot-Savart
law ensures that~$v_0$ belongs to~$\dH^{\frac 12}$. Let us introduce
the following family of spaces.
\begin{defi}
\label{definVorticiyspaces} {\sl For~$r$ in~$\bigl]\frac 32,
2\bigr]$, we denote by~$\cV^r$ the space of divergence free vector
fields with the vorticity of which belongs to~$L^{\frac 32}\cap
L^r$. }
\end{defi}
Let us remark that, if we denote \beq\label{nal} \al(r)\eqdefa
\frac1r-\f12\,\virgp \eeq the dual Sobolev
embedding~$L^r\hookrightarrow \dot H^{-3\al(r)}$  implies that the
vector field~$\cV^r$ is included into~$\dot  H^{\frac 12}\cap  \dot
H^{1-3\al(r)}$.

The purpose of this work is to generalize  the following result
proved by the first two authors in~\ccite{CZ5}.

\begin{thm}
\label{existenceomegaL3/2} {\sl Let us consider an initial data
$v_0$ in~$\cV^{\frac 32}$,  let us consider the  unique maximal
solution~$v$ associated with~$v_0$ given by
Theorem\refer{fujitakato+}.
 If its lifespan ~$T^\star$ is finite,  then  we have, for any~$p$ in~$]4,6[$ and any unit vector~$e$ in~$\R^3$,
$$
\int_0^{T^\star}\|(v(t)|e)_{\R^3}\|_{\dH^{\frac 1 2   +\frac 2 p }}^p\,dt=\infty.
$$
}
\end{thm}
 We refer to\ccite{CZ5} for a detailed introduction about the history of the results involving such ``anisotropic" norm for the description of blow up. The purpose of the present work is to drop the restriction on~$p$ supposing that the initial data is more regular. Namely, we prove the following theorem.
\begin{thm}
\label{thmain-1} {\sl Let us consider an initial data $v_0$
in~$\cV^{2}$,  let us consider the  unique maximal solutio  n~$v$
associated with~$v_0$ given by Theorem\refer{fujitakato+}.
 If its lifespan ~$T^\star$ is finite,  then  we have, for any~$p$ in~$]4,\infty[$ and any unit vector~$e$ in~$\R^3$,
$$
\int_0^{T^\star}\|(v(t)|e)_{\R^3}\|_{\dH^{\frac 1 2   +\frac 2 p }}^p\,dt=\infty.
$$
}
\end{thm}
Let us compare this theorem with the preceding one.
Theorem\refer{existenceomegaL3/2} deals with solution the regularity
of which is exactly at the scaling of $(NS)$.
Theorem\refer{thmain-1} deals with solutions which are continuous in
time with value in~$\dot H^{\frac 12} \cap \dot H^{1}$. But the blow
up condition is much better. Indeed, the bigger~$p$ is, the better
the blow up condition is. Let us recall that in the case when we
control the norm of all component, the blow up condition
about~$L^p_t(\dot H^{\frac 12 +\frac 2 p})$ is elementary for
finite~$p$.  The case when~$p$ is infinite, namely that fact that if
there is blow up in finite time~$T^\star$, then~$\ds
\limsup_{t\rightarrow T^\star} \|v(t)\|_{\dot H^{\frac 12}}$ is
infinite is a consequence of the work\ccite{ISS}  of L. Escauriaza,
G. Seregin and V. \u{S}ver\'{a}k. It is a deep result the proof of
which uses strongly the particular structure of the Navier-Stokes
equation.

Let us mention that the method presented here seems far away from
proving the limiting case when $p$ is infinite. Let us point out
that we have no idea about the following problem: let us assume
that for some unit vector~$e$ of~$\R^3$,
$\|(v_0|e)_{\R^3}\|_{H^{\frac 12}}$ is small with respect to some
universal constant, does it imply that there is no blow up for the
Fujita-Kato solution of~$(NS)$?

\section{ Ideas of the proof and structure of the paper}

First of all, let us mention that we do not prove directly Theorem\refer{thmain-1} but in fact the following one, which obviously implies Theorem\refer{thmain-1}.
\begin{thm}
\label{thmain}
 {\sl Let us consider~$r$ in~$[3/2,2[$ and an initial data~$v_0$ in~$\cV^r$. If the
  lifespan~$T^\star$ of the unique maximal solution~$v$ of~$(NS)$ given by Theorem\refer{fujitakato+}  is finite,  then  we have, for any~$p$ in~$\ds\Bigl]4,\frac {2r} {2-r} \Bigr [$
and any unit vector~$e$ in~$\R^3$,
\beq
\label{k.1}
\int_0^{T^\star}\|(v(t)|e)_{\R^3}\|_{\dH^{\frac 1 2   +\frac 2 p }}^p\,dt=\infty.
\eeq }
\end{thm}
The case when~$\ds r=\f32$ is exactly Theorem\refer{existenceomegaL3/2}.

\medbreak Let us explain the strategy of the proof of
Theorem\refer{thmain}. We  first remark that it makes no restriction
to assume that the unit vector~$e$ is the vertical
vector~$e_3\eqdefa (0,0,1).$  We follow essentially the same
strategy as that  in\ccite{CZ5} up to some differences due to the
fact that the regularity of  the solution~$v$ is higher than the one
given by the scaling of the equation.

\medbreak The first point consists in rewriting the homogeneous
incompressible Navier-Stokes equation  in terms of two unknowns:
\begin{itemize}
\item
the third component of the vorticity~$\Om$, which we denote by
$$
\om = \partial_1v^2-\partial_2v^1
$$
and which can be understood as the 2D vorticity for the vector field~$v^{\rm h} \eqdefa (v^1,v^2)$,

\item the quantity~$\partial_3v^3$ which is~$-\dive_hv^{\rm h}=- \partial_1v^1-\partial_2v^2$ because~$v$ is divergence free.
\end{itemize}
Immediate computations give
$$
(\wt {NS})
\quad\left\{
\begin{array}{c}
\partial_t\om+v\cdot\nabla\om -\D\om= \partial_3v^3\om +\partial_2v^3\partial_3v^1-\partial_1v^3\partial_3v^2,\\
\partial_t \partial_3v^3 +v\cdot\nabla\partial_3v^3-\D\partial_3v^3+\partial_3v\cdot\nabla v^3
=-\partial_3^2\D^{-1} \Bigl(\ds\sum_{\ell,m=1}^3 \partial_\ell
v^m\partial_mv^\ell\Bigr).
\end{array}
\right.
$$
Let us analyse this formulation of the Navier-Stokes system keeping
in mind that we already have control of ~$v^3$ in the
norm~$L^p_T\bigl(\dH^{\frac 1 2   + \frac 2 p }\bigr)$. Let us first
introduce  the notations \beq \label{a.1} \nabla_{\rm
h}^\perp=(-\p_2,\p_1),\quad\D_{\rm h}=\p_1^2+\p_2^2,\quad \vcurl
\eqdefa\nablah^\perp \D_{\rm h}^{-1} \om \andf \vdiv\eqdefa
-\nablah  \D_{\rm h}^{-1}
\partial_3v^3.
 \eeq
 Then we have, using the Biot-Savart's law in the horizontal variables
\beq
\label{a.1wrt} v^{\rm h}=\vcurl+\vdiv.
 \eeq
Let us concentrate on the equation on~$\omega$. As we have no a
priori control on~$\omega$, quadratic terms in this equation of
$(\wt{NS})$ seems dangerous. In fact,
 there is only one term of this type which is~$\vcurl\cdot\nablah\om$. A way to get rid of it is to
 use an energy type estimate and the divergence free condition
 on~$v$. Instead of working with scaling invariant norms as that in \cite{CZ5}, namely performing a
  $L^{\frac 32}$ energy estimate for~$\om,$ here we shall perform a $L^r$ energy estimate for $\om.$
This is based on the following lemma.
\begin{lem}
\label{estimpropagationLp}
 {\sl Let~$r$ be in~$]1,2[$ and~$a_0$ a
function in~$L^r$. Let us consider a function~$f$ in~$L^1_{\rm loc}(\R^+;L^r)$ and~$v$ a divergence free vector field in~$L^2_{\rm
loc}(\R^+;L^\infty)$ . If $a$ solves
$$
(T_v)\quad\left\{
\begin{array}{c} \partial_t a  -\D a +v\cdot\nabla a=f\\
a_{|t=0}=a_0
\end{array}
\right.
$$
then ~$|a|^{r/2}$ belongs to~$L^\infty_{\rm loc}(\R^+;L^2)\cap
L^2_{\rm loc}(\R^+;\dH^1)$ and \beq\label{tvestimate} \begin{split}
\frac1r\int_{\R^3} | a(t,x) |^rdx &+(r-1) \int_0^t \int_{\R^3}
|\nabla a(t',x)|^2|a(t',x)|^{r-2}
dx\,dt'\\
&=\frac1r\int_{\R^3} | a_0(x) |^r\,dx+ \int_0^t\int_{\R^3} f(t',x)
a(t',x) |a(t',x)|^{r-2}dx\,dt'. \end{split} \eeq }
\end{lem}
For the proof, see for instance\ccite{CZ5}, Lemma~3.1.

  \medbreak
  The  terms  on the right-hand side of the equation on~$\omega$ in $(\wt {NS})$  can be decomposed as~$\cL\om+F$ with
 $$
  \cL\om \eqdefa  \partial_3v^3\om +\partial_2v^3\partial_3v_{\rm curl}^1-\partial_1v^3\partial_3v_{\rm
curl}^2\andf F \eqdefa\partial_2v^3\partial_3v_{\rm div}^1-\partial_1v^3\partial_3v_{\rm
div}^2.
$$
These two terms are different. The  term~$\cL\om$ is linear with respect to~$\omega$ and thus can be estimated with quantities  related to scaling invariant space after some Gronwall lemma. The term~$F$ is a  forcing  term. It is quadratic with respect to~$v_3$   and will be estimated with one term related to scaling~$0$ and another term related to the scaling corresponding to the vorticity in~$L^r$.

\medbreak It remains   to examine  the second equation of~$(\wt
{NS}),$ which is
$$
 \partial_t \partial_3v^3 +v\cdot\nabla\partial_3v^3-\D\partial_3v^3+\partial_3v\cdot\nabla v^3
=-\partial_3^2\D^{-1} \Bigl(\ds\sum_{\ell,m=1}^3 \partial_\ell
v^m\partial_mv^\ell\Bigr).
$$
The main feature of this equation is that it contains only one
quadratic term with respect to~$\om$, namely the term
$$
-\partial_3^2\D^{-1} \Bigl(\ds\sum_{\ell,m=1}^2 \partial_\ell
v_{\rm curl}^m\partial_mv_{\rm curl}^\ell\Bigr).
$$
 Because we control~$v^3$ on some norm, a  way to get rid of this term  is to perform an energy estimate on~$\partial_3v^3$, namely an estimate on
 $$
\|\partial_3 v^3(t)\|_{\cH}
 $$
 for an adapted Hilbert space $\cH.$  Indeed, we  hope that if we control~$v^3,$ we can  control
terms of the type
$$
\bigl(\partial_3^2\D^{-1} (\partial_\ell
v_{\r  m curl}^m\partial_mv_{\rm curl}^\ell)\big| \partial_3v^3\bigr)_{\cH}
$$
with quadratic terms in~$\om$ and thus it fits
with~$\|\partial_3v^3\|_{\cH}^2$ so that we can hope to close the
estimate. Again here, the scaling  helps us for the choice of the
Hilbert space~$\cH$. The scaling  of~$\cH$ must be the scaling
of~$\dH^{-3\al(r)}$ for $\al(r)$ given by \eqref{nal}. Moreover,
because of the operator $\nablah\D_{\rm h}^{-1}$, it is natural to
measure horizontal derivatives and vertical derivatives differently.
This leads to the following definition.

\begin{defi}
\label{def2.1ad} {\sl For~$(s,s')$ in~$\R^2$, $\dH^{s,s'} $ denotes
the space of tempered distribution~$a$  such~that
 $$
\|a\|^2_{\dH^{s,s'}} \eqdefa \int_{\R^3} |\xi_{\rm
h}|^{2s}|\xi_3|^{2s'} |\wh a (\xi)|^2d\xi <\infty\with \xi_{\rm
h}=(\xi_1,\xi_2).
 $$
 For $\al(r)$ given by \eqref{nal} and ~$\theta$ in~$]0,\al(r)[$, we denote~$\cH^{\theta,r}\eqdefa \dH^{-3\al(r)+\theta,-\theta}$.
 }
 \end{defi}

 We want to emphasize  the fact that  anisotropy in the
regularity is highly related to the divergence free condition.
Indeed, let us consider a divergence free vector field~$w=(w^{\rm
h}, w^3)$ in~$\dH^{1-3\al(r)}$ and let us
estimate~$\|\partial_3w^3\|_{\cH^{\theta,r}}$. By definition of
the~$\cH^{\theta,r}$  norm, we have
 $$
 \|\partial_3w^3\|^2_{\cH^{\theta,r}} = A_L+A_H \with  A_L\eqdefa \int_{|\xi_{\rm h}|\leq |\xi_3|}
 |\xi_{\rm h}|^{-6\al(r)+2\th} |\xi_3|^{-2\th} |\cF(\partial_3w^3)(\xi)|^2 d\xi.
 $$
 In the case when~$|\xi_{\rm h}|\geq |\xi_3|$, since $\th\in ]0,\al(r)[,$ we write that
 \beno
 A_H  \leq  \int_{\R^3} |\xi_3|^{2(1-3\al(r))}\,|\wh w^3(\xi)|^2d\xi\leq \|w^3\|^2_{\dH^{1-3\al(r)}}.
 \eeno
 In the case when~$|\xi_{\rm h}|\leq |\xi_3|$, we use divergence free condition and write that   \beno
 A_L & \leq  &a  mp;  \int_{|\xi_{\rm h}|\leq |\xi_3|}
 |\xi_{\rm h}|^{-6\al(r)} |\cF(\dive_{\rm h} w^{\rm h})(\xi)|^2 d\xi\\
 & \leq &  \int_{\R^3}
 |\xi_{\rm h}|^{2(1-3\al(r))}\, | \widehat{w}^{\rm h}(\xi)|^2 d\xi = \|w^{\rm h}\|_{\dH^{1-3\al(r)}}^2.
 \eeno
 Thus  for any divergence  free vector field~$w$ in~$\dH^{1-3\al(r)}$,  we have
 \beq
 \label{initialdataHtheta}
 \|\partial_3w^3\|_{\cH^{\theta,r}} \leq C \|w\|_{\dH^{1-3\al(r)}}.
 \eeq
To use the space  efficiently in the proof, we need to rely them on
anisotropic Littlewood-Paley theory and also anisotropic Besov
spaces. This is the purpose of the third section.

\medbreak The first step of the proof of Theorem\refer{thmain} is
the following
 proposition:
  \begin{prop}
 \label{inegfondvroticity2D3D}
 {\sl
Let~$v_0$ be in~$\cV^r$;  let us consider a solution $v$ of~$(NS)$ given by Theorem \refer{fujitakato+}. Then for any $p$ in~$
 \bigl]4,\f{2r}{2-r}\bigr[$ and any~$\theta$ in~$]0,\al(r)[,$
 a constant~$C$ exists such that,  for any~$t<T^\star$,
\beq \label{a.10qp}
\begin{split}
\frac 1r \bigl\| \,\om_{\frac{r}2}(t)\bigr\|_{L^2}^{2}
 &+\frac{r-1}{r^2} \int_0^t\bigl\|\nabla \om_{\f{r}2}(t')\bigr\|_{L^2}^2\,dt'  \leq \biggl(\frac 1r  \| \,| \om_0|^{\frac{r}2}\|_{L^2}^{2}
\\
&\qquad{} + \Bigl(\int_0^t \|\p^2_3v^3(t')
\|^{2}_{\cH^{\theta,r}}\,dt'\Bigr)^{\frac {r}{2}}\biggr)\exp
\Bigl(C\int_0^t \|v^3(t')\|_{\dH^{\frac 1 2   +\frac 2 p }}^p dt'\Bigr).
\end{split}
\eeq }
\end{prop}
 Here and in all  that follows,  for scalar function~$a$ and
for~$\al$ in the interval~$]0,1[$, we always  denote
\beq
\label{a.10wl}
a_\al \eqdefa  \frac a {|a|} |a|^\al.
\eeq
Up to some technical difficulties, the proof
of this proposition follows essentially the lines  of the analogous proposition in\ccite {CZ5}; it is the purpose of the fourth section.

\medbreak
Next we want to control~$\|\p^2_3v^3\|_{L^2_t(\cH^{\theta,r})}$. As
already explained, a way to get rid of the only quadratic term
in~$\om$, namely
$$
-\partial_3^2\D^{-1}\Bigl( \sum_{\ell,m=1}^2 \partial_\ell v_{\rm curl}^m\partial_m v_{\rm curl}^\ell\Bigr)
$$
is to perform  an energy estimate for the norm~$\cH^{\theta,r}$.
\begin{prop}
 \label{estimadivhaniso}
 {\sl
Let~$v_0$ be in~$\cV^r$;  let us consider a solution $v$ of~$(NS)$ given by Theorem\refer{fujitakato+}. For any~$p$ in~$\bigl]4,\f{2r}{2-r}\bigr[$ and $\th$ in
 $\bigl]3\al(r)-\frac 2 p , \al(r) \bigr[,$
 a constant~$C$ exists such that for any~$t<T^\star,$ we have
\beq
\label{b.4bqp}
\begin{split}
& \|\p_3v^3(t)\|_{\cH^{\theta,r}}^2
 +\int_0^t\|\na\p_3v^3(t')\|_{\cH^{\theta,r}}^2\,dt' \leq C \exp \Bigl(C\int_0^t
\|v^3(t')\|^{p}_{\dH^{\frac 1 2   +\frac 2 p }}\,dt' \Bigr)\\
 &\qq  uad\qquad{}\times\biggl(\|\Om_0\|_{L^{r}}^2
 + \int_0^t \Bigl(\|v^3(t')
\|_{\dH^{\frac 1 2   +\frac 2 p }}\bigl\|\om_{\f{r}2}(t')\bigr\|_{L^2}^{2\left(2\al(r)+\frac 1 p \right)}\bigl\|\na\om_{\f{r}2}(t')\bigr\|_{L^2}^{\frac 2 {p'}}
\\
&\qquad\qquad\qquad\qquad\qquad{}+\|v^3(t')\|_{\dH^{\frac 1 2   +\frac 2 p }}^2\bigl\|\om_{\f{r}2}(t')\bigr\|_{L^2}^{4\left(\al(r)+\frac 1 p \right)}
\bigl\|\na\om_{\f{r}2}(t')\bigr\|_{L^2}^{2\left(1-\frac 2 p \right)}\Bigr)\,dt'\biggr).
\end{split}
\eeq Here and in all that follows, $p'$ denotes the conjugate number
of $p$ so that $\ds \frac1{p'}=1-\frac1p\,\cdotp$}
\end{prop}
The proof of this proposition is different from the one of the
analogous proposition in\ccite{CZ5}. In the framework of that
article, only laws of product  of quantities related to scaling $0$
were used. Here, two different scalings are involved and we  do use
the structure of the transport  term  ~$v\cdot \nabla $ to prove
propagation estimate in a quasi-linear spirit. The term~$(v^{\rm
h}\cdot \nabla_{\rm h}
\partial_3v^3|\partial_3v^3)_{\cH^{\theta,r}}$ requires a particular
care (see forthcoming Lemma\refer{estimatedivhdemoeq3}). The proof
of Proposition\refer{estimadivhaniso}  is the purpose of the fifth
section.

\medbreak A non-standard Gronwall type argument allows to deduce
from Propositions\refer{inegfondvroticity2D3D}
and\refer{estimadivhaniso}
 that we control the quantities
\beq
\label{listquantities}
\|\om\|_{L^\infty_t(L^{r})},\quad
 \int_0^t\|\nabla \om_{\f r 2}(t')\|_{L^2}^2\,dt'\,,\ \int_0^t \|\p^2_3v^3(t')
\|^{2}_{\cH^{\theta,r}}\,dt' \andf \int_0^t\|v^3(t')\|_{\dH^{\frac 1
2   +\frac 2 p }}^pdt'\,. \eeq Let us also point out  that these
quantities have different scaling; the quantity
$$
\int_0^t\|v^3(t')\|_{\dH^{\frac 1 2   +\frac 2 p }}^pdt'
$$
is scaling invariant and   the quantities
$$
\|\om\|_{L^\infty_t(L^{r})},\quad
 \int_0^t\|\nabla \om_{\f r 2}(t')\|_{L^2}^2\,dt'\andf\ \int_0^t \|\p^2_3v^3(t')
\|^{2}_{\cH^{\theta,r}}\,dt'
$$
are the scaling of the norm~$L^\infty_t(\dot H^{1-3\alpha(r)})$.
Biot-Savart law in the horizontal variable allows to prove   that all
the above quantities  in\refeq{listquantities} prevents the solution
$v$ of $(NS)$ from blowing up.  The details of all this is the
purpose of  the last section.

\setcounter{equation}{0} \section{Non linear inequalities and
Littlewood-Paley analysis}
 In this section,  we recall or prove estimates that will be useful later on and recall the basics of anisotropic Littlewood-Paley theory.
 As a warm up, let us establish some
Sobolev type inequalities which involve the regularities of
~$a_{\frac r2}$ and~$\nabla a_{\frac r2}$ in~$L^2$ which are
relevant to Lemma\refer{estimpropagationLp}.

\begin{lem}
\label{BiotSavartomega}
{\sl For $r$ in~ $]3/2,2[,$ we have
\beq
\label{estimbasomega34}
 \|\nabla a\|_{L^{r}} \lesssim \bigl\|\nabla a_{\frac{r}2}\bigr\|_{L^2}
 \bigl\|a_{\frac{r}2}  \bigr\|_{L^2}^{\frac 2 {r}   -1}.
 \eeq
Moreover, for~$s$  in $\ds[-3\al(r)\,\virgp\, 1-\al(r)]$, we have
\beq
\label{estimbasomega34.0}
\|a\|_{\dH^s} \leq C
\|a_{\frac{r}2}\|_{L^2} ^{1-\al(r)-s}  \|\nabla
a_{\frac{r}2}\|_{L^2} ^{3\al(r)+s} .
 \eeq }
\end{lem}
\begin{proof}
Notice that due to \eqref{a.10wl}, \beno |\na a |&=&\f2r|\na
a_{\f{r}2}|\,|a|^{1-\f{r}2}\\
&=& \f2r|\na a_{\f{r}2}|\,|a_{\f{r}2}|^{\frac 2 {r}   -1}, \eeno then we get
\eqref{estimbasomega34} by using H\"older inequality. The dual
Sobolev inequality claims that \beq \label{BiotSavartomegademoeq1}
\|a\|_{\dH^{-3\al(r)}} \leq C \|a\|_{L^{r}}= C
\|a_{\frac{r}2}\|_{L^2} ^{\frac 2 {r}   }. \eeq Moreover, using again that $\ds
|\na a |=\f2r|\na a_{\f{r}2}|\,|a_{\f r2}|^{\frac 2 {r}   -1},$ H\"older
inequality implies that \ben
\|\nabla a \|_{L^{\frac{3r}{1+r}}} & \leq &  \f2r\|\nabla a_{\frac{r}2}\|_{L^2} \|\,a_{\f r2}\|_{L^{6}}^{\frac 2 {r}   -1}\nonumber\\
& \lesssim &    \|\nabla a_{\frac{r}2}\|_{L^2}^{\frac 2 {r}
}.\label{qw.3} \een Since $r<2,$ we have $\f{3r}{1+r}<2.$ Then
Theorem 2.40 of \cite{BCD} ensures that
$L^{\f{3r}{1+r}}\hookrightarrow\dot{B}^0_{\f{3r}{1+r},2},$ which
along with Bernstein's inequality implies
$$
\|a\|_{\dH^{1-\al(r)}}\lesssim \|\na a\|_{\dot{B}^0_{\f{3r}{1+r},2}}
\lesssim  \|\nabla a_{\f{r}2}\|_{L^2}^{\frac 2 {r}   },
$$
from which and \eqref{BiotSavartomegademoeq1}, we conclude the proof
of \eqref{estimbasomega34.0} and hence the lemma by using
interpolation inequality between~$\dH^s$ S  obolev spaces.
\end{proof}
As we shall use the anisotropic Littlewood-Paley theory, we recall
the functional space framework we are going to use in this section.
As in\ccite{CDGG}, \ccite{CZ1}, \ccite{Pa02} and \ccite{CZ5}, the
definitions of the spaces we are going to work with require
anisotropic dyadic decomposition   of the Fourier variables. Let us
recall from \cite{BCD} that \beq
\begin{split}
&\Delta_k^{\rm h}a=\cF^{-1}(\varphi(2^{-k}|\xi_{\rm h}|)\widehat{a}),\qquad
\Delta_\ell^{\rm v}a =\cF^{-1}(\varphi(2^{-\ell}|\xi_3|)\widehat{a}),\\
&S^{\rm h}_ka=\cF^{-1}(\chi(2^{-k}|\xi_{\rm h}|)\widehat{a}),
\qquad\ S^{\rm v}_\ell a =
\cF^{-1}(\chi(2^{-\ell}|\xi_3|)\widehat{a})
 \quad\mbox{and}\\
&\Delta_ja=\cF^{-1}(\varphi(2^{-j}|\xi|)\widehat{a}),
 \qquad\ \
S_ja=\cF^{-1}(\chi(2^{-j}|\xi|)\widehat{a}), \end{split}
\label{1.0}\eeq where $\xi_{\rm h}=(\xi_1,\xi_2),$ $\cF a$ and
$\widehat{a}$ denote the Fourier transform of the distribution $a,$
$\chi(\tau)  $ and~$\varphi(\tau)$ are smooth functions such that
 \beno
&&\Supp \varphi \subset \Bigl\{\tau \in \R\,/\  \ \frac34 \leq
|\tau| \leq \frac83 \Bigr\}\andf \  \ \forall
 \tau>0\,,\ \sum_{j\in\Z}\varphi(2^{-j}\tau)=1,\\
&&\Supp \chi \subset \Bigl\{\tau \in \R\,/\  \ \ |\tau|  \leq
\frac43 \Bigr\}\quad \ \ \ \andf \  \ \, \chi(\tau)+ \sum_{j\geq
0}\varphi(2^{-j}\tau)=1.
 \eeno

\begin{defi}\label{def4.1}
{\sl  Let $(p,r)$ be in~$[1,+\infty]^2$ and~$s$ in~$\R$. Let us consider~$u$ in~${\mathcal
S}_h'(\R^3),$ which means that $u$ is in~$\cS'(\R^3)$ and satisfies~$\ds\lim_{j\to-\infty}\|S_ju\|_{L^\infty}=0$. We set
$$
\|u\|_{\dB^s_{p,r}}\eqdefa\big\|\big(2^{js}\|\Delta_j
u\|_{L^{p}}\big)_j\bigr\|_{\ell ^{r}(\ZZ)}.
$$
\begin{itemize}

\item
For $s<\frac{3}{p}$ (or $s=\frac{3}{p}$ if $r=1$), we define $
\dB^s_{p,r}(\R^3)\eqdefa \big\{u\in{\mathcal S}_h'(\R^3)\;\big|\; \|
u\|_{\dB^s_{p,r}}<\infty\big\}.$

\item
If $k$ is  a positive integer and if~$\frac{3}{p}+k\leq
s<\frac{3}{p}+k+1$ (or $s=\frac{3}{p}+k+1$ if $r=1$), then we
define~$ \dB^s_{p,r}(\R^3)$  as the subset of distributions $u$
in~${\mathcal S}_h'(\R^3)$ such that $\partial^\beta u$ belongs to~$
\dB^{s-k}_{p,r}(\R^3)$ whenever $|\beta|=k.$
\end{itemize}
}
\end{defi}
We remark that in the particular case when $p=r=2,$ $\dB^s_{p,r}$
coincides with the classical homogeneous Sobolev spaces $\dH^s$.
Likewise, we can also define Besov spaces in the inhomogeneous
context  (see  \ccite{BCD} for instance).

\medbreak
The description of the regularity of~$\om_{r-1}$ in terms of Besov spaces will be useful. This is done thanks to the following lemma (see\ccite{CZ5}, Lemma 5.1).
\begin{lem}
 \label{puisancealphaBesov}
 {\sl Let~$(s,\al)$ be in~$]0,1[^2$ and~$(p,q)$ in~$[1,\infty]^2$.  We consider a function~$G$ from $\R$ to~$\R$ which is H\"olderian
  of exponent~$\al$. Then
 for any~$a$ in the Besov space~$ \dB^s_{p,q},$ one has
 $$
 \|G(a)\|_{\dB^{\al s}_{\frac p\al,\frac q\al} }\lesssim \|G\|_{C^\al} \bigl(\|a\|_{\dB^s_{p,q}}\bigr)^\al\with
\|G\|_{C^\al} \eqdefa \sup_{r\not =r' } \frac
{|G(r)-G(r')|}{|r-r'|^\al} \,\cdotp
 $$
 }
 \end{lem}


\medbreak

Similar to Definition \ref{def4.1}, we can also define the
homogeneous anisotropic Besov space.

\begin{defi}\label{anibesov}
{\sl Let us define the space $\bigl(\dB^{s_1}_{p,q_1}\bigr)_{\rm
h}\bigl(\dB^{s_2}_{p,q_2}\bigr)_{\rm v}$ as the space of
distribution in~$\cS'_h$  such that
$$
\|u\|_{\bigl(\dB^{s_1}_{p,q_1}\bigr)_{\rm
h}\bigl(\dB^{s_2}_{p,q_2}\bigr)_{\rm v}}\eqdefa \biggl(\sum_{k\in\Z}
2^{q_1ks_1} \Bigl(\sum_{\ell\in\Z}2^{q_2\ell s_2}\|\D_k^{\rm
h}\D_\ell^{\rm
v}u\|_{L^p}^{q_2}\Bigr)^{{q_1}/{q_2}}\biggr)^{1/{q_1}}
$$
is finite.
}
\end{defi}
We remark that when $p=q_1=q_2=2,$ the anisotropic Besov space
$\bigl(\dB^{s_1}_{p,q_1}\bigr)_{\rm
h}\bigl(\dB^{s_2}_{p,q_2}\bigr)_{\rm v}$ coincides with the
classical homogeneous anisotropic Sobolev space $\dH^{s_1,s_2}$ and
thus the space~$\bigl(\dB^{-3\al(r)+\theta}_{2,2}\bigr)_{\rm
h}\bigl(\dB^{-\theta}_{2,2}\bigr)_{\rm v}$ is the space~$\cH^{\theta,r}$
defined in Definition\refer{def2.1ad}. Let us also remark that in
the case when
 $q_1$ is   different from~$q_2$, the order of summation is important.

\medbreak
 For the
convenience of the readers, we recall the following anisotropic
Bernstein type lemma from \cite{CZ1, Pa02}:

\begin{lem}
\label{lemBern}
{\sl Let $\cB_{h}$ (resp.~$\cB_{v}$) a ball
of~$\R^2_{h}$ (resp.~$\R_{v}$), and~$\cC_{h}$ (resp.~$\cC_{v}$) a
ring of~$\R^2_{h}$ (resp.~$\R_{v}$); let~$1\leq p_2\leq p_1\leq
\infty$ and ~$1\leq q_2\leq q_1\leq \infty.$ Then there holds:

\smallbreak\noindent If the support of~$\wh a$ is included
in~$2^k\cB_{h}$, then
\[
\|\partial_{x_{\rm h}}^\alpha a\|_{L^{p_1}_{\rm h}(L^{q_1}_{\rm v})}
\lesssim 2^{k\left(|\al|+2\left(1/{p_2}-1/{p_1}\right)\right)}
\|a\|_{L^{p_2}_{\rm h}(L^{q_1}_{\rm v})}.
\]
If the support of~$\wh a$ is included in~$2^\ell\cB_{v}$, then
\[
\|\partial_{x_3}^\beta a\|_{L^{p_1}_{\rm h}(L^{q_1}_{\rm v})}
\lesssim 2^{\ell\left(\beta+(1/{q_2}-1/{q_1})\right)} \|
a\|_{L^{p_1}_{\rm h}(L^{q_2}_{\rm v})}.
\]
If the support of~$\wh a$ is included in~$2^k\cC_{h}$, then
\[
\|a\|_{L^{p_1}_{\rm h}(L^{q_1}_{\rm v})} \lesssim
2^{-kN}\sup_{|\al|=N} \|\partial_{x_{\rm h}}^\al a\|_{L^{p_1}_{\rm
h}(L^{q_1}_{\rm v})}.
\]
If the support of~$\wh a$ is included in~$2^\ell\cC_{v}$, then
\[
\|a\|_{L^{p_1}_{\rm h}(L^{q_1}_{\rm v})} \lesssim 2^{-\ell N}
\|\partial_{x_3}^N a\|_{L^{p_1}_{\rm h}(L^{q_1}_{\rm v})}.
\]
}
\end{lem}

\medbreak As a corollary of Lemma \ref{lemBern}, we have the
following inequality, if~$1\leq p_2\leq p_1$,
\beq
\label{inclusionSobolevtypeaniso}
\|a\|_{\bigl(\dB^{s_1-2\left(\frac 1{p_2}- \frac 1{p_1}\right)}_{p_1,q_1}\bigr)_{\rm h}
\bigl(\dB^{s_2-\left(\frac 1 {p_2}-\frac 1
{p_1}\right)}_{p_1,q_2}\bigr)_{\rm v}} \lesssim
\|a\|_{\bigl(\dB^{s_1}_{p_2,q_1}\bigr)_{\rm
h}\bigl(\dB^{s_2}_{p_2,q_2}\bigr)_{\rm v}}.
\eeq

\medbreak

 To consider the
product of a distribution in the isotropic Besov space with a
distribution in the anisotropic Besov space, we need the following
result which allows to embed  isotropic Besov spaces into the
anisotropic ones.
\begin{lem}[Lemma 4.2 of \cite{CZ5}]\label{embeda}
{\sl Let $s$ be a positive real number and $(p,q)$ in~$ [1,\infty]$
with~$p$ greater than or equal to~$q.$ Then one has \beno
\|a\|_{L^p_{\rm h}\bigl((\dB^s_{p,q})_{\rm v}\bigr)}\lesssim
\|a\|_{\dB^s_{p,q}}. \eeno }
\end{lem}

\begin{lem}[Lemma 4.3 of \cite{CZ5}]
\label{isoaniso}
 {\sl For any~$s$ positive and any~$\theta$ in~$]0,s[$,
we have
$$
\|f\|_{(\dB^{s-\theta}_{p,q})_{\rm h}(\dB^{\theta}_{p,1})_{\rm v}}
\lesssim \|f\|_{\dB^{s}_{p,q}}.
$$
}
\end{lem}

One of the main motivation of using anisotropic Besov space is the
proof of the following Proposition \ref{BiotSavartBesovaniso}, which
extends $r=3/2$ in \cite{CZ5} to general $r$ in~$ ]3/2,2[.$

\begin{lem}
\label{interpolHtheta} {\sl  Let us consider $\th$ in $]0,3\alpha(r)[$
and ~$\beta$
in~$]0, 1/2   [$.
Then we have
$$
\|a\|_{\bigl(\dB^{0}_{2,1}\bigr)_{\rm
h}\bigl(\dB^{1-3\alpha(r)-\beta}_{2,1}\bigr)_{\rm v}}\lesssim
\|a\|_{\cH^{\theta,r}}^{\beta}\|\nabla a\|_{\cH^{\theta,r}}^{1-\beta}.
$$
}
\end{lem}
\begin{proof}
By definition of~$\|\cdot\|_{\bigl(\dB^{0}_{2,1}\bigr)_{\rm
h}\bigl(\dB^{1-3\alpha(r)-\beta}_{2,1}\bigr)_{\rm v}}$, we have \ben
\|a\|_{\bigl(\dB^{0}_{2,1}\bigr)_{\rm h}\bigl(\dB^{1-3\alpha(r)-\beta}_{2,1}\bigr)_{\rm v}}& =& H_L(a)+V_L(a)\with \nonumber\\
\label{interpolHthetademoeq1} H_L(a) & \eqdefa & \sum_{k\leq \ell}
\|\D_k^{\rm h}\D^{\rm v}_\ell a\|_{L^2}
2^{\ell(1-3\alpha(r)-\beta)} \andf\\
 V_L(a) & \eqdefa & \sum_{k> \ell} \|\D_k^{\rm h}\D^{\rm v}_\ell a\|_{L^2}  2^{\ell(1-3\alpha(r)-\beta)}. \nonumber
 \een
 In order to estimate~$H_L(a),$ we classically estimate differently  high and low vertical frequencies
  which are here the dominant ones. Using Lemma\refer{lemBern}, we  write that for any~$N$ in~$\ZZ$,
$$
H_L(a)  \lesssim    \sum_{k\leq \ell\leq N} \|\D_k^h\D^{\rm v}_\ell
a\|_{L^2} 2^{\ell(1-3\alpha(r)-\beta)}+ \sum_{\substack{k\leq
\ell\\\ell>N}} \|\D_k^h\D^{\rm v}_\ell \partial_3a\|_{L^2} 2^{-\ell
(3\alpha(r)+\beta)}.
$$
By definition of the norm of~$\cH^{\theta,r}$, we get
$$
H_L(a)  \lesssim  \|a\|_{\cH^{\theta,r}}  \sum_{k\leq \ell\leq N}
2^{k(3\alpha(r)-\theta)} 2^{\ell(1-3\alpha(r)-\beta+\theta)}+
\|\partial_3a\|_{\cH^{\theta,r}}\sum_{\substack{k\leq \ell\\\ell>N}}
2^{k(3\alpha(r)-\theta)} 2^{-\ell(3\alpha(r)+\beta-\theta)}.
$$
The hypothesis on~$(\beta,\theta)$ imply that
\beno H_L(a)  & \lesssim &
 \|a\|_{  \cH^{\theta,r}}  \sum_{\ell\leq N}  2^{\ell(1-\beta)}+
 \|\partial_3a\|_{\cH^{\theta,r}} \sum_{\ell>N}   2^{-\ell\beta}\\
& \lesssim &  \|a\|_{\cH^{\theta,r}}  2^{N(1-\beta)}+
\|\partial_3a\|_{\cH^{\theta,r}}  2^{-N\beta}.
\eeno
Choosing~$N$ such that~$\ds
2^N\sim\frac{\|\partial_3a\|_{\cH^{\theta,r}}}{\|a\|_{\cH^{\theta,r}}}$
gives
 \beq
 \label{interpolHthetademoeq2} H_L(a)\lesssim
\|a\|_{\cH^{\theta,r}}^\beta\|\partial_3a\|_{\cH^{\theta,r}}^{1-\beta}.
\eeq
The term~$V_L(a)$ is estimated along the same lines. In fact,
we  get, by using again Lemma\refer{lemBern}, that
\beno
V_L(a) &\lesssim &
\sum_{\ell<k\leq N} \|\D_k^h\D^{\rm v}_\ell a\|_{L^2}
2^{\ell
(1-3\alpha(r)-\beta)} + \sum_{\substack{ \ell<k\\
k>N}}\|\D_k^h\D^{\rm v}_\ell
\nabla_{\rm h}a\|_{L^2} 2^{\ell(1-3\alpha(r)-\beta)}2^{-k}\\
& \lesssim &  \|a\|_{\cH^{\theta,r}} \sum_{\ell <k\leq
N}2 ^{k(3\alpha(r)-\theta)}
 2^{\ell(1-3\alpha(r)-\beta+\th)} \\
 &&\qquad\qquad\qquad+
 \|\nabla_{\rm h} a\|_{\cH^{\theta,r}}\sum_{\substack{\ell\leq k\\ k>N}}2^{-k(1-3\alpha(r)+\theta)} 2^{\ell(1-3\alpha(r)-\beta+\th)}\\
 & \lesssim &  \|a\|_{\cH^{\theta,r}}  2^{N(1-\beta)}+  \|\nabla_{\rm h}a\|_{\cH^{\theta,r}}  2^{-N\beta}.
\eeno
Choosing~$N$ such that~$\ds 2^N\sim \frac{\|\nabla_{\rm h}
a\|_{\cH^{\theta,r}}}{\|a\|_{\cH^{\theta,r}}}$ yields
$$
V_L(a)\lesssim \|a\|_{\cH^{\theta,r}}^\beta\|\nabla_{\rm h}
a\|_{\cH^{\theta,r}}^{1-\beta}.
$$
Together with\refeq{interpolHthetademoeq1}
and\refeq{interpolHthetademoeq2}, this ensures the lemma.
\end{proof}

\begin{prop}
\label{BiotSavartBesovaniso}
{\sl
 Let~$v$ be a divergence free vector
field. Let us consider $\theta$ in $]0,3\al(r)[$ and~$\beta$ in~$]0,
1/2   [.$
Then we have
$$
\|v^{\rm h}\|_{\bigl(\dB^{1}_{2,1}\bigr)_{\rm
h}\bigl(\dB^{1-3\al(r)-\beta}_{2,1}\bigr)_{\rm v}}\lesssim \bigl\|\,
\om_{\frac{r}2}\bigr\|_{L^2} ^{2\al(r)+\beta} \bigl\|\nabla
\om_{\frac{r}2}\bigr\|_{\cH^{\theta,r}} ^{1-\beta}
+\|\partial_3v^3\|_{L^2}^\beta \|\nabla
\partial_3v^3\|_{\cH^{\theta,r}}^{1-\beta}.
$$
}
\end{prop}
\begin{proof}
Using horizontal Biot-Savart law\refeq {a.1} and
Lemma\refer{lemBern}, we have \beq
\label{BiotSavartBesovanisodemoeq1} \|v^{\rm
h}\|_{\bigl(\dB^{1}_{2,1}\bigr)_{\rm h}
\bigl(\dB^{1-3\al(r)-\beta}_{2,1}\bigr)_{\rm v}} \lesssim
\|\om\|_{\bigl(\dB^{0}_{2,1}\bigr)_{\rm
h}\bigl(\dB^{1-3\alpha(r)-\beta}_{2,1}\bigr)_{\rm v}} +
\|\pa  rtial_3v^3\|_{\bigl(\dB^{0}_{2,1}\bigr)_{\rm
h}\bigl(\dB^{1-3\alpha(r)-\beta}_{2,1}\bigr)_{\rm v}}.
 \eeq
 Applying Lemma \refer{lemBern} and Lemma \refer{isoaniso} gives
 \ben
\|\om\|_{\bigl(\dB^{0}_{2,1}\bigr)_{\rm
h}\bigl(\dB^{1-3\alpha(r)-\beta}_{2,1}\bigr)_{\rm v}}
& \lesssim &
\|\om\|_{\bigl(\dB^{\frac23\alpha(r)}_{\frac{3r}{1+r},1}\bigr)_{\rm h}
\bigl(\dB^{1-\f83\alpha(r)-\beta}_{\frac{3r}{1+r},1}\bigr)_{\rm v}}
\nonumber\\
\label{BiotSavartBesovanisodemoeq2}& \lesssim &
\|\om\|_{\dB^{1-2\alpha(r)-\beta}_{\frac{3r}{1+r},1}}.
 \een
 Now let us estimate~$\|\om\|_{\dB^s_{\f{3r}{1+r},1}}$  in terms
of~$\bigl\|\om_{\frac{r}2}\bigr\|_{L^2}$
and~$\bigl\|\na\om_{\frac{r}2}\bigr\|_{L^{2}}$.
 For $s$ in~$ ]-2\alpha(r), 1[$  and any positive integer $N,$ which we shall choose hereafter, we write that
\beno \|\om\|_{\dB^s_{\f{3r}{1+r},1}} &=&
\sum_{j\leq N}2^{  js}\|\D_j\om\|_{L^{\f{3r}{1+r}}}+\sum_{j>N}2^{js}\|\D_j\om\|_{L^{\f{3r}{1+r}}}\\
& \lesssim &
\sum_{j\leq N}2^{j(s+2\alpha(r))}\|\D_j\om\|_{L^{r}}+\sum_{j>N}2^{j(s-1)}\|\D_j\na\om\|_{L^{\f{3r}{1+r}}}\\
& \lesssim &
2^{N(s+2\alpha(r))}\|\om\|_{L^{r}}+2^{N(s-1)}\|\na\om\|_{L^{\f{3r}{1+r}}}.
\eeno
 Choosing $N$ such that $2^N\sim \Bigl(\f{\|\na
\om\|_{L^{\f{3r}{1+r}}}}{\|\om\|_{L^{r}}}\Bigr)^{\f1{1+2\alpha(r)}}$
yields
$$
\|\om\|_{\dB^s_{\f{3r}{1+r},1}}\lesssim
\|\om\|_{L^{r}}^{\f{1-s}{1+2\alpha(r)}}\|\na\om\|_{L^{\f{3r}{1+r}}}^{\frac {s+2\alpha(r)} {1+2\al(r)}}.
$$
Due to \eqref{nal}, $ 1+2\alpha(r)=\frac 2 {r}   ,$ then using the above
inequality with~$s=\ds 1-2\alpha(r) -\beta\in ]-2\alpha(r), 1[,$
\eqref{BiotSavartBesovanisodemoeq2}, and \eqref{qw.3} gives \ben
\|\om\|_{\bigl(B^{0}_{2,1}\bigr)_{\rm
h}\bigl(B^{1-3\alpha(r)-\beta}_{2,1}\bigr)_{\rm v}} &\lesssim&
\|\om\|_{L^r}^{\f{r}2\bigl(2\alpha(r)+\beta\bigr)}\|\na\om\|_{L^{\f{3r}{1+r}}}^{\f{r}2  (1-\beta)}\nonumber\\
 &\lesssim& \bigl\|\,
\om_{\frac{r}2}\bigr\|_{L^2} ^{2\alpha(r)+\beta} \bigl\|\nabla
\om_{\frac{r}2}\bigr\|_{L^2} ^{1-\beta}. \label{b.8} \een

The application of Lemma  \ref{interpolHtheta} together
with\refeq{BiotSavartBesovanisodemoeq1} and\refeq{b.8} leads to Proposition \ref{BiotSavartBesovaniso}.
\end{proof}

To study product laws between distributions in the anisotropic Besov
spaces, we need to  modify the isotropic para-differential
decomposition of  Bony \cite{Bo81} to the setting of anisotropic
version. We first recall the isotropic para-differential
decomposition from \cite{Bo81}: let $a$ and~$b$ be in~$
\cS'(\R^3)$, \beq \label{pd}\begin{split} &
ab=T(a,b)+\bar{T}(a,b)+ R(a,b)\with\\
& T(a,b)=\sum_{j\in\Z}S_{j-1}a\Delta_jb, \quad
\bar{T}(a,b)=T(b,a), \andf\\
&R(a,b)=\sum_{j\in\Z}\Delta_ja\tilde{\Delta}_{j}b,\quad\hbox{with}\quad
\tilde{\Delta}_{j}b=\sum_{\ell=j-1}^{j+1}\D_\ell a. \end{split} \eeq
Sometimes we shall use Bony's decomposition for both horizontal and
vertical variables simultaneously.

Finally let us recall the following product laws in the anisotropic
Besov spaces from \cite{CZ5}:

\begin{lem}[Lemma 4.5 of \cite{CZ5}]
\label{lem2.2qw} {\sl Let $q\geq 1,$ $p_1\geq p_2\geq 1$ with
$1/{p_1}+1/{p_2}\leq 1,$ and $s_1< 2/{p_1},$ $ s_2< 2/{p_2}$ (resp.
$s_1\leq 2/{p_1},$ $ s_2\leq 2/{p_2}$ if $q=1$) with $s_1+s_2>0.$
Let $\sigma_1< 1/{p_1},$ $ \sigma_2< 1/{p_2}$ (resp. $\sigma_1\leq
1/{p_1},$ $ \sigma_2\leq 1/{p_2}$ if $q=1$) with
$\sigma_1+\sigma_2>0$. Then for $a$ in~$
\bigl(\dB^{s_1}_{p_1,q}\bigr)_{\rm
h}\bigl(\dB^{\s_1}_{p_1,q}\bigr)_{\rm v}$ and~$b$
in~$\bigl(\dB^{s_2}_{p_2,q}\bigr)_{\rm
h}\bigl(\dB^{\s_2}_{p_2,q}\bigr)_{\rm v}$ , the product~$ab$ belongs
to~$ \bigl(\dB^{s_1+s_2-2/{p_2}}_{p_1,q}\bigr)_{\rm
h}\bigl(\dB^{\sigma_1+\sigma_2-{1}/{p_2}}_{p_1,q}\bigr)_{\rm v},$
and \beno \|a b\|_{\bigl(\dB^{s_1+s_2-2/{p_2}}_{p_1,q}\bigr)_{\rm
h}\bigl(\dB^{\sigma_1+\sigma_2-{1}/{p_2}}_{p_1,q}\bigr)_{\rm
v}}\lesssim \|a\|_{\bigl(\dB^{s_1}_{p_1,q}\bigr)_{\rm
h}\bigl(\dB^{\s_1}_{p_1,q}\bigr)_{\rm v}}\|b\|_{\bigl(\dB^{s_2}_{p_2,q}\bigr)_{\rm
h}\bigl(\dB^{\s_2}_{p_2,q}\bigr)_{\rm v}}. \eeno}
\end{lem}


\setcounter{equation}{0}
\section{Proof of the  estimate for the horizontal vorticity}
\label{proofestimateomega}

 The purpose of this section to present the proof of Proposition \ref{inegfondvroticity2D3D}. Let us recall  the first  equation of our reformulation~$(\wt {NS})$ of the incompressible Navier-Stokes equation
 which is
 $$
 \partial_t\om+v\cdot\nabla\om -\D\om=   \partial_3v^3\om +\partial_2v^3\partial_3v^1-\partial_1v^3\partial_3v^2.
 $$
 As already explained in the second section, we decompose the right-hand side term as a sum of three terms.  Hence by virtue of \eqref{tvestimate}, we obtain
 \beq
\label{theoxplosaniscalingdemoeq1}
\begin{split}
&\frac 1r \bigl\|\om_{\frac{r}2}(t) \bigr\|_{L^2}^2  +
\frac{4(r-1)}{r^2}\int_0^t \bigl\|\nabla
\om_{\frac{r}2}(t')\bigr\|_{L^2}^2\,dt' = \frac1r
\bigl\||\om_0|^{\frac{r}2}
\bigr\|_{L^2}^2 +\sum_{\ell=1}^3 F_\ell(t) \with\\
&F_1(t) \eqdefa  \int_0^t \int_{\R^3} \p_3v^3 |\om|^{r}\,dx\,dt'\,,\\
&F_2(t) \eqdefa  \int_0^t \int_{\R^3}\bigl(\p_2v^3\p_3v_{\rm curl}^1-\p_1v^3\p_3v_{\rm curl}^2\bigr) \om_{r-1}\,dx\,dt'\andf\\
&F_3(t) \eqdefa  \int_0^t \int_{\R^3}\bigl(\p_2v^3\p_3v_{\rm
div}^1-\p_1v^3\p_3v_{\rm div}^2 \bigr) \om_{r-1}\,dx\,dt',
\end{split}
 \eeq
where~$v_{\rm curl}^{\rm h}$ (resp.~$v_{\rm div}^{\rm h}$)
corresponds to the horizontal divergence free (resp. curl free) part of
the horizontal vector~$v^{\rm h}=(v^1,v^2),$ which is given by
\eqref{a.1},  and where~$\om_{r-1}\eqdefa |\om|^{r-2}\,\om$.

Let us start with the easiest term~$F_1$. 
 We first get, by using
integration by parts, that \beno |F_1(t)|&\leq&
r\int_0^t\int_{\R^3}|v^3(t',x)|\,|\p_3\om(t',x)|\,|\om(t',x)|^{r-1}\,dx\,dt'\\
&\leq&
r\int_0^t\int_{\R^3}|v^3(t',x)|\,|\p_3\om(t',x)|\,|\om_{\f{r}2}(t',x)|^{\frac 2 {r'}   }\,dx\,dt'.
\eeno Using that
$$
\frac {p-2} {3p} + \frac1r +
\frac{2pr-3p+2r}{6p(r-1)}\times\frac 2 {r'} =1,
$$
we apply H\"older inequality to get
$$
|F_1(t)|\leq r
\int_0^t\|v^3(t')\|_{L^{\f{3p}{p-2}}}\|\p_3\om(t')\|_{L^{r}}\bigl\|\om_{\frac{r}2}(t')\bigr\|_{L^{\f{6p(r-1)}{2pr-3p+2r}}}^{\frac 2 {r'}   }\,dt'.
$$
As~$p$ is in~$\ds\bigl]4,\f{2r}{2-r}\bigr[$, we have that
$\ds r'\f{p-2}{2p}$ belongs to~$]0,1[.$ Then  Sobolev embedding and
interpolation inequality implies that
$$
\bigl\|\om_{\frac{r}2}(t')\bigr\|_{L^{\f{6p(r-1)}{2pr-3p+2r}}}
\lesssim
\bigl\|\om_{\frac{r}2}  (t')\bigr\|_{\dH^{r'\frac{(p-2)}{2p}}}\lesssim
\bigl\|\om_{\frac{r}2}(t')\bigr\|_{L^2}^{\frac{2r-p(2-r)}{2p(r-1)}}
\bigl\|\nabla \om_{\f{r}2}(t')\bigr\|_{L^2}^{r'\frac{p-2}{2p}}.
$$
Using\refeq{estimbasomega34} of Lemma\refer{BiotSavartomega}, this gives
$$
{ |F_1(t)| \lesssim \int_0^t\|v^3(t')\|_{\dH^{\frac 1 2   +\frac 2 p }}
\bigl\|\p_3\om_{\frac{r}2}(t')\bigr\|_{L^2}
\bigl\|\om_{\frac{r} 2}(t')\bigr\|_{L^2}^{\frac 2 {r}   -1} }
{
\bigl\|\na\om_{\frac{r}2}(t')\bigr\|_{L^2}^{1-\frac 2 p }\bigl\|\om_{\frac{r}2}(t')\bigr\|_{L^2}^{1-2\left(\frac 1 {r}   -\frac 1 p \right)}\,dt'.
}
$$
Applying convex inequality, we obtain
 \ben
|F_1(t)| & \lesssim &
\int_0^t\|v^3(t')\|_{\dH^{\frac 1 2   +\frac 2 p }}\bigl\|\om_{\frac{r}2}(t')\bigr\|_{L^2}^{\frac 2 p }\bigl\|\na\om_{\frac{r}2}(t')\bigr\|_{L^2}^{\frac 2 {p'}} \,dt' \nonumber\\
 \label{a.9}   & \leq &
\f{r-1}{r^2}\int_0^t\bigl\|\na\om_{\frac{r}2}(t')\bigr\|_{L^2}^2\,dt'
+C\int_0^t\|v^3(t')\|_{\dH^{\frac 1 2   +\frac 2 p }}^p\bigl\|\om_{\frac{r}2}(t')\bigr\|_{L^2}^{2}\,dt'.
 \een

\medbreak The other two terms in \eqref{theoxplosaniscalingdemoeq1}
require a refined way of the description of the regularity
of~$\om_{\frac r2}$ and demand a detailed study  of the anisotropic
operator~$\nabla_{\rm h}\D_{\rm h}^{-1}$ associated with the
Biot-Savart's law in horizontal variables.   Now we first modify
Lemma 5.2 of  \cite{CZ5} to the following one.

 \begin{lem}
 \label{anositropiclemma}
{\sl Let  $\theta$ be in~$]0,\al(r)[$ for $\al(r)$ given by
\eqref{nal}, and~$\s,$ ~$s$ be such that
$$
\s\in \Bigl]\frac {r'} 4,1\Bigr[ \andf  s=\frac 1 2   +1-\frac {2\s}{r'}\, \cdotp
$$Then we have
 \beq
 \label{b.1}
 \Bigl| \int_{\R^3}
\partial_{\rm h}\D_{\rm h}^{-1}f
\partial_{\rm h}a \,\om_{r-1} dx\Bigr|
 \lesssim  \min\bigl\{ \|f\|_{L^{r}},  \| f\|_{\cH^{\theta,r}}\bigr\}  \|a\|_{\dH^s} \bigl\|\om_{\frac{r}2}\bigr\|_{\dH^\s}^{\frac 2 {r'}   }
\eeq
for $\cH^{\theta,r}$ given by Definition \ref{def2.1ad}.}
 \end{lem}

 \begin{proof}
 Let us observe that~$\ds \om_{r-1} =G(\om_{\frac{r}2})$ with~$G(z)\eqdefa z|z|^{-2\al(r)}$. Using Lemma\refer{puisancealphaBesov}, we obtain
 \beq
 \label{anositropiclemmademoeq1}
\bigl\|\om_{r-1}\bigr\|_{\dB^{\frac {2\s} {r'} }_{r'   ,r'   }}
\lesssim \bigl\|\om_{\frac{r}2}\bigr\|_{\dH^\s}^{\frac 2 {r'}   }.
 \eeq
 Let us study the product~$\partial_{\rm h}a\om_{r-1}$. Using Bony's decomposition \eqref{pd} and the Leibnitz formula, we
 write
 \beno
 \partial_{\rm h} a\, \om_{r-1} & = & T({\partial_{\rm h} a}, \om_{r-1}) +R(\partial_{\rm h}a,\om_{r-1})+
 T({\om_{r-1}}, \partial_{\rm h}a)\\
 & = & \partial_{\rm h}T({\om_{r-1}}, a) +A(a,\om)\with \\
 A(a,\om) &\eqdefa&
  T({\partial_{\rm h} a}, \om_{r-1}) +R(\partial_{\rm h}a,\om_{r-1})
 -T({ \partial_{\rm h}\om_{r-1}},a).
 \eeno
W  e first get, by using Lemma \ref{lemBern}, that
 \beno
\|\D_jT({\om_{r-1}}, a)\|_{L^2}& \lesssim &\sum_{|j-j'|\leq
4}\|S_{j'-1}\om_{r-1}\|_{L^\infty}\|\D_{j'}a\|_{L^2}\\
&\lesssim&  \sum_{|j-j'|\leq 4}
2^{j'\bigl(\frac 3 {r'}   -\frac {2\s} {r'}   \bigr)}\bigl\|\om_{r-1}\bigr\|_{\dB^{\frac {2\s} {r'} }_{r'   ,r'   }}
c_{j',2}
2^{-j's}\|a\|_{\dH^s}\\
&\lesssim & c_{j,2}2^{-j\bigl(s-\frac 1 {r'}
(3-2\s)\bigr)}\bigl\|\om_{r-1}\bigr\|_{\dB^{\frac {2\s} {r'} }_{r'
,r'   }} \|a\|_{\dH^s}. \eeno Here and in what follows, we always
denote $\bigl(c_{j,r}\bigr)_{j\in\Z}$ to be a generic element in the
sphere of~$\ell^r(\Z)$.  Then together with
\eqref{anositropiclemmademoeq1}, the above inequality ensures that
\beq\label{poi}
 \bigl\| T({\om_{r-1}}, a) \|_{\dH^{s-\frac 1 {r'}   (3-2\s)} } \lesssim \|a\|_{\dH^s} \bigl\|\om_{\f{r}2}\bigr\|  _{\dH^\s}^{\frac 2 {r'}   }.
 \eeq
Using that the operator~$ \partial_{\rm h}^2\D_{\rm h}^{-1}$ is a bounded Fourier multiplier and  the
 dual Sobolev embedding~$L^{r}\hookrightarrow \dH^{-3\al(r)}$, we get by
 taking~$\ds s=\frac 32-\frac {2\s} {r'}   $ in \eqref{poi} that
\ben \Bigl| \int_{\R^3} \partial_{\rm h}\D_{\rm h}^{-1}f
\partial_{\rm h} T({\om_{r-1}}, a)\, dx\Bigr| & =&
 \Bigl|\int_{\R^3}
\partial^2_{\rm h}\D_{\rm h}^{-1}f T({\om_{r-1}}, a) \,dx\Bigr|\nonumber \\
&  \leq &
\|f\|_{\dot H^{-3\al(r)}}\|T({\om_{r-1}}, a)\|_{\dot H^{3\al(r)}} \nonumber \\
  \label{anositropiclemmademoeq2}&  \lesssim &
  \|f\|_{L^{r}} \|a\|_{\dH^s}
  \bigl\|\om_{\frac{r}2}\bigr\|_{\dH^\s}^{\frac 2 {r'}   }.
 \een
 In the case of the anisotropic norm, recalling that~$\cH^{\theta,r}= \dH^{-3\al(r)+\theta,-\theta}$, and using Lemma\refer{isoa niso},  we write
\ben
 \Bigl| \int_{\R^3}
\  partial^2_{\rm h}\D_{\rm h}^{-1}f  T({\om_{r-1}}, a)\, dx\Bigr|  & \leq &
 \|f\|_{\cH^{\theta,r}} \|T({\om_{r-1}}, a)\|_{\dH^{3\al(r)-\theta,\theta}}\nonumber \\
 & \lesssim &  \|f\|_{\cH^{\theta,r}} \|T({\om_{r-1}}, a)\|_{\dH^{3\al(r)}}\nonumber \\
\label{anositropiclemmademoeq2b}& \lesssim &  \|f\|_{\cH^{\theta,r}}
\|a\|_{\dH^s} \bigl\|\om_{\frac{r}2}\bigr\|_{\dH^\s}^{\frac 2 {r'}   }.
  \een

 Now let us take into account the anisotropy induced by the operator~$\partial_{\rm h}\D_{\rm h}^{-1}$.
  Hardy-Littlewood-Sobolev inequality implies that~$\partial_{\rm h}\D_{\rm h}^{-1} f$ belongs to~$L^{r}_{\rm v}(L^{\f{2r}{2-r}}_{\rm h})$
   if $f$ is in~$ L^{r}.$
  So that it amounts  to prove that~$A(a,\om)$ belongs to~$L^{r'   }_{\rm v}(L^{\frac{2r}{3r-2}}_{\rm h}),$ which is simply an anisotropic Sobolev type embedding.
   Because of ~$s=\frac 1 2   +1-\frac {2\s} {r'}   <1$, we get, by using Lemma \ref{lemBern}, Inequality\refeq{anositropiclemmademoeq1} and H\"older inequality with~$2$ and~$r'$, that
\beno
\|\D_jT({\partial_{\rm h}a},\om_{r-1})\|_{L^{\f{2r}{3r-2}}}
&\lesssim &
\sum_{|j'-j|\leq  4}\|S_{j'-1}\partial_{\rm h}a\|_{L^2}\|\D_{j'}\om_{r-1}\|_{L^{r'   }}\\
&  \lesssim &\sum_{|j'-j|\leq 4}c_{j',2}c_{j',r'   }2^{  j'(1-s)}
\|a\|_{\dH^s}2^{-2j'\frac \s {r'}   }\bigl\|\om_{\f{r}2}\bigr\|_{\dH^\s}^{\frac 2 {r'}   }\\
&\lesssim&
c_{j,\f{2r}{3r-2}}2^{-\f{j}2}\|a\|_{\dH^s}\bigl\|\om_{\frac{r}2}\bigr\|_{\dH^\s}^{\frac 2 {r'}   }.
\eeno
Along the same lines,  we have
\beno
\|\D_jR(\p_{\rm h}a,\om_{r-1})\|_{L^{\f{2r}{3r-2}}}&\lesssim& \sum_{j'\geq
j-3}\|\D_{j'}\p_{\rm
h}a\|_{L^2}\|\wt{\D}_{j'}\om_{r-1}\|_{L^{r'   }}\\
&\lesssim&\sum_{j'\geq j-
   3}c_{j',2}c_{j',r'   }2^{j'(1-s)}\|a\|_{\dH^s}2^{-2j'\frac \s {r'}   }\bigl\|\om_{\f{r}2}\bigr\|_{\dH^\s}^{\frac 2 {r'}   }\\
&\lesssim
&c_{j,\f{2r}{3r-2}}2^{-\f{j}2}\|a\|_{\dH^s}\bigl\|\om_{\frac{r}2}\bigr\|_{\dH^\s}^{\frac 2 {r'}   },
\eeno and \beno \|\D_jT(\p_{\rm
h}\om_{r-1},a)\|_{L^{\f{2r}{3r-2}}}&\lesssim& \sum_{|j'-j|\leq
4}\|S_{j'-1}\p_{\rm
h}\om_{r-1}\|_{L^{  r'   }}\|\wt{\D}_{j'}a\|_{L^2}\\
&\lesssim
&c_{j,\f{2r}{3r-2}}2^{-\f{j}2}\|a\|_{\dH^s}\bigl\|\om_{\frac{r}2}\bigr\|_{\dH^\s}^{\frac 2 {r'}   }.
\eeno
 This leads to
 \beq\label{b.2}
 \|A(a,\om)\|_{\dB^{\frac 1 2   }_{\frac{2r}{3r-2},\frac{2r}{3r-2}}} \lesssim  \|a\|_{\dH^s}\bigl\|\om_{\frac{r}2}\bigr\|_{\dH^\s}^{\frac 2 {r'}   }.
 \eeq
While it follows from Lemma \ref{embeda} that
 $$
\dB^{\frac 1 2   }_{\frac{2r}{3r-2},\frac{2r}{3r-2}} \hookrightarrow
L^{\frac{2r}{3r-2}}_{\rm
h}\biggl(\Bigl(\dB^{\frac 1 2   }_{\frac{2r}{3r-2},\frac{2r}{3r-2}}\Bigr)_{\rm v}\biggr).
 $$
Sobolev type embedding theorem (see for instance Theorem~2.40
of  \ccite{BCD}) claims that
 $$
 \dB^{\frac 1 2   }_{\frac{2r}{3r-2},\frac{2r}{3r-2}}(\R) \hookrightarrow \dB^0_{r'   ,2}(\R) \hookrightarrow L^{r'   }(\R).
 $$
 Moreover, since $r'   >\f{2r}{  3r-2},$ we have
 $$ L^{\f{2r}{3r-2}}_{\rm h}(L^{r'   }_{\rm v})\hookrightarrow L^{r'   }_{\rm v}(L^{\f{2r}{3r-2}}_{\rm h}).$$
 As a consequence, by virtue of \eqref{b.2}, we obtain
 \beno
 \Bigl|
\int_{\R^3}
\partial_{\rm h}\D_{\rm h}^{-1}f A(a,\om) \, dx\Bigr| & \lesssim &
\|\partial_{\rm h}\D_{\rm h}^{-1}f\|_{L^{r}_{\rm v}(L^{\f{2r}{2-r}}_{\rm h})}\|A(a,\om)\|_{L^{r'   }_{\rm v}(L^{\f{2r}{3r-2}}_{\rm h})}\\
&\lesssim&
\|f\|_{L^{r}}\|A(a,\om)\|_{\dB^{\frac 1 2   }_{\frac{2r}{3r-2},\frac{2r}{3r-2}}}\\
& \lesssim &
\|f\|_{L^{r}}\|a\|_{\dH^s}\bigl\|\om_{\frac{r}2}\bigr\|_{\dH^\s}^{\frac
2 {r'}   }, \eeno which together with
\refeq{anositropiclemmademoeq2} gives rise to \beno \Bigl|
\int_{\R^3}
\partial_{\rm h}\D_{\rm h}^{-1}f
\partial_{\rm h}a \,\om_{r-1} dx\Bigr|
 \lesssim   \|f\|_{L^{r}}  \|a\|_{\dH^s} \bigl\|\o  m_{\frac{r}2}\bigr\|_{\dH^\s}^{\frac 2
 {r'}}. \eeno

\medbreak
 In order to prove the remaining inequality of \eqref{b.1}, we observe that
 $$
\| \nabla_{\rm h}\D_{\rm h}^{-1} f\|_{\dH^{1-3\al(r)+\theta,-\theta}
}\lesssim \|f\|_{\dH^{-3\al(r)+\theta,-\theta} }=\|f\|_{\cH^{\theta,r}}.
 $$
 Thus thanks to \eqref{b.2}, for $\th$ given by the lemma, it amounts  to prove  that
 \beq\label{b.3} \dB^{\frac 1 2   }_{\frac{2r}{3r-2},\frac{2r}{3r-2}} \hookrightarrow
 \dH^{-1+3\al(r)-\theta,\theta}. \eeq
 As a matter of fact, using Lemma \ref{isoaniso} and Lemma \ref{lemBern}, we have, for any~$\gamma$ in~$]\ds0,1/2   [$,
\beno \dB^{\frac 1 2   }_{\frac{2r}{3r-2},\frac{2r}{3r-2}}
&\hookrightarrow& \bigl(\dB^{\frac 1 2
-\gamma}_{\frac{2r}{3r-2},\frac{2r}{3r-2}}\bigr)_{\rm
h}\bigl(\dB^{\gamma}_{\frac{2r}{3r-2},\frac{2r}{3r-2}}\bigr)_{\rm v}
\andf\\
\bigl(\dB^{\frac 1 2
-\gamma}_{\frac{2r}{3r-2},\frac{2r}  {3r-2}}\bigr)_{\rm
h}\bigl(\dB^{\gamma}_{\frac{2r}{3r-2},\frac{2r}{3r-2}}\bigr)_{\rm v}
&\hookrightarrow&
  \bigl(\dB^{\frac 1 2   -\gamma-\frac 2 {r'}   }_{2,2}\bigr)_{\rm h}\bigl(\dB^{\gamma-\frac 1 {r'} }_{2,2})\bigr)_{\rm v}\\
  &\hookrightarrow&
 \dH^{-\gamma+\frac 2 {r}   -\frac 3 2,\gamma-1+\frac 1 {r}   }.
\eeno Let us choose $\ds\theta\eqdefa\gamma-\frac 1 {r'}   .$ Then
since $\th <\al(r),$ we have $\ds\gamma<\frac 1 2   ,$ which ensures
\eqref{b.3}.  This completes the proof of the lemma.
 \end{proof}

\medbreak The estimate of~$F_2(t)$ uses the Biot-Savart's law in the
horizontal variables  (namely\refeq{a.1}) and
Lemma\refer{anositropiclemma} with~$f=\partial_3\om$,~$a=v^3$. This
gives for any time~$t<T^\star$ and $\s$ in~$
\bigl]\f{r}{4(r-1)},1\bigr[$ that
 $$
 \longformule{
I_\om(t)\eqdefa\Bigl|\int_{\R^3}\bigl(\p_2v^3(t,x)\p_3v_{\rm
curl}^1(t,x)-\p_1v^3(t,x)\p_3v_{\rm curl}^2(t,x) \bigr)
\om_{r-1}(t,x)\,dx\Bigr| }{{}  \lesssim  \|\p_3\om(t) \|_{L^{r}} \|
v^3(t)\|_{\dH^{\frac 3 2-\frac {2\s} {r'}    }}
\bigl\|\om_{\frac{r}2}(t)\bigr\|_{\dH^\s}^{\frac 2 {r'}   }. }
 $$
By virtue of \eqref{estimbasomega34} and of the interpolation
inequalities between~$L^2$ and~$\dH^1$, we thus obtain \beno
 I_\om(t)
&\lesssim
&\|v^3(t)\|_{\dH^{\frac 1 2   +2\left(\frac 1 2   -\frac {\s} {r'}    \right)}}\bigl\|\om_{\f{r}2}(t)\bigr\|_{L^2}^{\frac 2 {r}   -1}\bigl\|\na\om_{\f{r}2}(t)\bigr\|_{L^2}\\
&&\qquad\qquad\qquad\qquad\qquad\qquad\times\bigl\|\om_{\f{r}2}(t)\bigr\|_{L^2}^{\frac 2 {r'} (1-\s)  }\bigl\|\na\om_{\f{r}2}(t)\bigr\|_{L^2}^{\frac {2\s} {r'}    }\\
&\lesssim& \|v^3(t)\|_{\dH^{\frac 1 2   +2\left(\frac 1 2   -\frac
{\s} {r'}    \right)}}
 \bigl\|\om_{\frac{r}2}(t)\bigr\|_{L^2} ^{2\bigl(\frac 1 2   -\frac {\s} {r'}    \bigr)}\bigl\|\nabla\om_{\frac{r}2}(t)\bigr\|_{L^2}^{2\bigl(\frac 1 2   +\frac {\s} {r'}    \bigr)}.
 \eeno
Choosing~$\ds \s = \f{(p-2)r}{2p(r-1)}$, which is between
$\ds\f{r'}{4}$ and~$1$ because~$p$ is between~$4$
and~$\ds \f{2r}{2-r}$, gives
$$
I_\om(t)\lesssim \|v^3(t)\|_{\dot H^{\frac 1 2   +\frac 2 p }}
\bigl\|\om_{\frac{r}2}(t)\bigr\|_{L^2} ^{ \frac 2 p
}\bigl\|\nabla\om_{\frac{r}2}(t  )\bigr\|_{L^2}^{2\left(1 -\frac 1 p
\right)} .
$$
Then  by using convexity inequality and time integration, we get
 \beq
 \label{a.10}
 |F_2(t)| \leq \frac{r-1}{r^2} \int_0^t\bigl\|\nabla \om_{\frac{r}2}(t')\bigr\|_{L^2}^2
 dt' +C \int_0^t  \|v^3(t')\|_{\dH^{\frac 1 2   +\frac 2 p }}^p\bigl\|\om_{\frac{r}2}(t')\bigr\|_{L^2}
 ^{2}\,dt'.
 \eeq

In order to estimate~$F_3(t)$, we write
$$
\longformule{ F_3(t) =- \int_0^t
\int_{\R^3}\Bigl(\p_2v^3(t',x)(\p_1\D_{\rm h}^{-1}\p_3^2v^3)(t',x) }
{{} -\p_1v^3(t',x)(\p_2\D_{\rm h}^{-1}\p_3^2v^3)(t',x) \Bigr)\,
\om_{r-1}(t',x)dxdt'. }$$ Since ~$\ds \s =
\f{(p-2)r}{2p(r-1)},$~$\ds \frac 1 p =\frac 1 2    -\frac {\s} {r'}
$, thanks to interpolation inequality between Sobolev spaces,  we
get, by applying Lemma\refer{anositropiclemma}
with~$f=\partial_3^2v^3$ and ~$a=v^3,$
 that
\beno
 |F_3(t)| & \lesssim & \int_0^t \|\p^2_3v^3(t')
\|_{\cH^{\theta,r}}\| v^3(t')\|_{\dH^{\frac 3 2-\frac {2\s} {r'} }}
\bigl\|\om_{\frac{r}2}(t')\bigr\|_{\dH^\s}^{\frac 2 {r'}   }\,dt'\\
& \lesssim & \int_0^t\|\p^2_3v^3(t') \|_{\cH^{\theta,r}}\|
v^3(t')\|_{\dH^{\frac 1 2   +2\left(\frac 1 2   -\frac {\s} {r'}
\right)}}\bigl\|\om_{\frac{r}2}(t')\bigr\|_{L^2}^{\frac 2 {r'}
(  1-\s) }
\bigl\|\na\om_{\frac{r}2}(t')\bigr\|_{L^2}^{\frac {2\s} {r'}    }\,dt'\\
& \lesssim & \int_0^t\|\p^2_3v^3(t') \|_{\cH^{\theta,r}}
\|v^3(t')\|^{p\al(r)}_{\dH^{\frac 1 2   +\frac 2 p }}\\
&&\qquad\qquad\qquad{}\times\bigl(\|v^3(t')\|^{p}_{\dH^{\frac 1 2   +\frac 2 p }}\bigl\|\om_{\frac{r}2}(t')\bigr\|^2_{L^2}\bigr)^{
\frac 1 p -\al(r)}
\bigl\|\na\om_{\frac{r}2}(t')\bigr\|_{L^2}^{2\left(\frac 1 2   -\frac 1 p \right)}\,dt'.
\eeno As we have
$$
\frac 1 2+ \al(r) +\biggl(\frac 1 p-\al(r)\biggr) +\biggl(\frac 12
-\frac1p\biggr)=1,
$$
 applying H\"older inequality ensures that
$$
\longformule{
 |F_3(t)| \lesssim \Bigl(\int_0^t \|\p^2_3v^3(t')\|^{2}_{\cH^{\theta,r}}\,dt'\Bigr)^{\frac 1{2}}
  \Bigl(\int_0^t \|v^3(t')\|^{p}_{\dH^{\frac 1 2   +\frac 2 p }}\,dt' \Bigr)^{\al(r)}
} { {}\times  \Bigl(\int_0^t \|v^3(t')\|_{\dH^{\frac 1 2   ;  +
\frac 2 p }}^p\bigl\|\om_{\frac{r}2}(t')\bigr\|_{L^2}^{2}\,dt'\Bigr)^{\frac1{p}-\al(r)}
\Bigl(\int_0^t \bigl\|\nabla \om_{\frac{r}2}(t')\bigr\|^2_{L^2}
dt'\Bigr)^{\f12-\f1p}. }
$$
Applying the convexity inequality leads to
 \beq
\label{a.8}
\begin{split}
 |F_3(t)| \leq\,& \frac{r-1}{r^2}  \int_0^t
\bigl\|\nabla \om_{\frac r 2}(t')\bigr\|^2_{L^2} dt' +C \int_0^t
\|v^3(t')\|_{\dH^{\frac 1 2   +
\frac 2 p }}^p\bigl\|\om_{\frac{r}2}(t')\bigr\|_{L^2}
 ^{2}\,dt' \\
&\qquad\qquad\qquad{}+ C\Bigl(\int_0^t \| v^3(t')\|^{p}_{\dH^{\frac
1 2   +\frac 2 p }}\,dt' \Bigr)^{1-\frac{r}2} \Bigl(\int_0^t
\|\p^2_3v^3(t') \|^{2}_{\cH^{\theta,r}}\,dt'\Bigr)^{\frac {r}{2}}.
\end{split} \eeq

\begin{proof}[Conclusion of the proof to
Proposition\refer{inegfondvroticity2D3D}] Resuming the Estimates
\eqref{a.9}, \eqref{a.10} and \eqref{a.8} into
\eqref{theoxplosaniscalingdemoeq1}, we obtain \beno \frac 1 r
\bigl\|\om_{\frac r 2}(t) \bigr\|_{L^2}^2  + \frac{r-1}{r^2}\int_0^t
\|\nabla \om_{\frac r 2}(t')\|_{L^2}^2\,dt' & &
\\
&&
\ \ \ \ \ \ \ \ \ \ \ \ \ \ \ \ \ \ \ \ \ \ \ \ \ \ \ \ \ \ \ \ \ \ \ \ \ \ \ \ \ \ \ \ \ \ \ \ \ \ \ \ \ \ \ \ \ \ \ \ \ \ \ \ \ \ \ \ \ \ \ \ \ \ \ \ \ \ \ \ \ \ \ \ \ \ \ \ \
\leq \frac 1 r \bigl\||\om_0|^{r/2} \bigr\|_{L^2}^2{}+
C\Bigl(\int_0^t \| v^3(t')\|^{p}_{\dH^{\frac 1 2   +\frac 2 p
}}\,dt' \Bigr)^{1-\f{r}2} \Bigl(\int_0^t \|\p^2_3v^3(t')
\|^{2}_{\cH^{\theta,r}}\,dt'\Bigr)^{\f{r}2}\\
&&
\quad {}+C\int_0^t\| v^3(t')\|_{\dH^{\frac 1 2   +\frac 2 p }}^{p} \bigl\|\om_{\frac
r 2}(t')\bigr\|_{L^2}^{2}\,dt'  . \eeno
 Inequality~\eqref{a.10qp} follows from Gronwall lemma once
notice that ~$x^{\frac 14} e^{Cx}\lesssim e^{C'x}$ for~$C'>C$.
\end{proof}


\setcounter{equation}{0}
\section{Proof of the  estimate for the second vertical derivatives of $v^3$}
\label{Sectionestimatedivh}

In this section, we shall present the proof of Proposition
\ref{estimadivhaniso}. Let $\cH^{\theta,r}$ be given by
Definition\refer{def2.1ad}.  We  get, by taking the $\cH^{\theta,r}$ inner
product of the $\p_3v^3$ equation of $(\wt{NS})$ with $\p_3v^3,$
that \beq \label{b.4pu}
\begin{split}
\f12\f{d}{dt}\|\p_3v^3(t)\|_{\cH^{\theta,r}}^2+&\|\na\p_3v^3(t)\|_{\cH^{\theta,r}}^2 =-\sum_{n=1}^3 \bigl(Q_n(v,v) \, |\, \p_3v^3\bigr)_{\cH^{\theta,r}}\with\\
 Q_1(v,v) & \eqdefa \bigl(\Id+\partial_3^2\D^{-1} \bigr)(\partial_3v^3)^2 +\partial_3^2\D^{-1}
  \biggl(\sum_{\ell,m=1}^2\partial_\ell v^m\partial_m  v^\ell\biggr) \,,\\
 Q_2(v,v) & \eqdefa  \bigl(\Id+2\partial_3^2\D^{-1} \bigr)\biggl(\sum_{\ell=1}^2 \partial_3 v^\ell
 \partial_\ell v^3\biggr)\andf\\
 Q_3(v,v) & \eqdefa v\cdot \nabla\partial_3 v^3.
 \end{split}
 \eeq
 The  estimate involving~$Q_1$ relies on    the following lemma. Let us point out that this term~$Q_1$ contains  terms which are quadratic with respect to~$v^{\rm h}_{\rm curl}$.
 \begin{lem}
\label{estimdivhlemme1}
{\sl Let~$A$ be a bounded Fourier multiplier. If~$p$  and~$\theta$ satisfy
 \beq
\label{Conditionenergyaniso} 0<\th<\f12-\f1p\,\virgp
 \eeq
 then we have
 $$
 \bigl|\bigl(A(D) (fg) \,|\,\partial_3v^3\bigr)_{\cH^{\theta,r}}\bigr| \lesssim \|f\|_{\dH^{\frac 1 2   -3\al(r)+\th,\frac 1 2   -\frac 1 p -\theta} }
\|g\|_{\dH^{\frac 1 2   -3\al(r)+\th,\frac 1 2   -\frac 1 p -\theta} } \|v^3\|_{\dH^{\frac 1 2   +
\frac 2 p }}.
 $$
 }
 \end{lem}
 \begin{proof}
 Let us first observe that, for any couple~$(\al,\beta)$ in~$\R^2$, we have, thanks to
Cauchy-Schwartz inequality, that, for any real valued function~$a$
and~$b$, \ben \bigl|(a|b)  _{\cH^{\theta,r}}\bigr|  &  =  &
\Bigl|\int_{\R^3} |\xi_{\rm h}|^{-6\al(r)+2\theta-\alpha}
 |\xi_3|^{-\beta-2\theta} \wh a(\xi)|\xi_{\rm h}|^{\alpha}  \,|\xi_3|^{\beta} \wh b(-\xi)d\xi\Bigr|\nonumber\\
\label{estimdivhlemme1demoeq1}
 & \leq & \|a\|_{\dH^{-6\al(r)+2\theta-\alpha, -\beta-2\theta}} \|b\|_{\dH^{\alpha, \beta}} .
\een  As~$A(D)$ is  a bounded Fourier multiplier, applying
\eqref{estimdivhlemme1demoeq1} with~$\al= 0$ and~$\ds\beta= -\frac 1 2
+\frac 2 p $, we obtain
 \beq
 \label{estimdivhlemme1demoeq2}
\bigl|\bigl(A(D) (fg) \,|\,\partial_3v^3\bigr)_{\cH^{\theta,r}}\bigr|
\lesssim \|fg\|_{\dH^{-6\al(r)+2\theta,\frac 1 2   -\frac 2 p -2\theta}}
\|\partial_3v^3\|_{\dH^{0,-\frac 1 2   +\frac 2 p }} .
\eeq
 Because~$\dH^{s,s'} = \bigl(\dB^s_{2,2}\bigr)_{\rm h} \bigl(\dB^{s'}_{2,2}\bigr)_{\rm v
 },$ $r$ is in~$]3/2,2[,$ $3\al(r)$ is less than~$\frac 1 2  $< BR>  and thanks to Condition\refeq{Conditionenergyaniso},  law of products of Lemma\refer{lem2.2qw} implies in particular  that
 $$
  \|fg\|_{\dH^{-6\al(r)+2\theta,\frac 1 2   -\frac 2 p -2\theta}}  \lesssim \|f\|_{\dH^{\frac 1 2   -3\al(r)+\theta, \frac 1 2   -\frac 1 p -\theta}}
  \|g\|_{\dH^{\frac 1 2   -3\al(r)+\theta,\frac 1 2   -\frac 1 p -\theta}}.
 $$
 Due to Lemma \ref{embeda}, we have
 $$
 \|\partial_3v^3\|_{\dH^{0,-\frac 1 2   +\frac 2 p }} \lesssim \|v^3\|_{\dH^{0,\frac 1 2   + \frac 2 p }}\leq \|v^3\|_{\dH^{\frac 1 2   +\frac 2 p }}
 $$
 and thus  the lemma is proved.
 \end{proof}

 \medbreak
  Because both~$\p_3^2\D^{-1}$ and $\partial_{\rm h}^2\D_{\rm h}^{-1}$ are bounded Fourier
  multipliers,
 applying Lemma \ref{estimdivhlemme1} with~$f$ and~$g$ of the form~$\partial_{\rm h} v^{\rm h}_{\rm curl}$ or~$\partial_{\rm h} v^{\rm h}_{\rm div}$
  or with~$f=g=\partial_3v^3$ gives,
 $$
\bigl| \bigl(Q_1(v,v) \, |\, \p_3v^3\bigr)_{\cH^{\theta,r}}\bigr| \lesssim
\|v^3\|_{\dH^{\frac 1 2   +\frac 2 p }} \Bigl(
 \|\om\|_{\dH^{\frac 1 2   -3\al(r)+\theta, \frac 1 2   -\frac 1 p -\theta}}^2 +\|\partial_3v^3\|_{\dH^{\frac 1 2   -3\al(r)+\theta, \frac 1 2   -\frac 1 p -\theta}}^2\Bigr).
 $$
Because $p>4,$ $r>4/3,$ we have $\f1p+\f1r<1,$ and
$\th<\al(r)<1/2-1/p$ so that the
Condition\refeq{Conditionenergyaniso} is satisfied. Then we get, by
using Lemma \ref{isoaniso} and Lemma\refer{BiotSavartomega}, that
\beq \label{r  q.2}
 \|\om\|_{\dH^{\frac 1 2   -3\al(r)+\theta,
\frac 1 2   -\frac 1 p -\theta}}\leq \|\om\|_{\dH^{1-3\al(r)- \frac 1 p }}\lesssim
\bigl\|\, \om_{\frac{r}2}\bigr\|_{L^2}^{2\al(r)  +\frac 1 p  } \bigl\|\nabla
\om_{\frac{r}2}\bigr\|_{L^2} ^{\frac 1 {p'}   }.
\eeq
While it follows from Definition \ref{def2.1ad} that
\beno
\|a\|_{\dH^{\frac 1 2   -3\al(r)+\theta,
\frac 1 2   - \frac 1 p -\theta}}^2&=&\int_{\R^3}|\xi_{\rm
h}|^{1-6\al(r)+2\th}|\xi_3|^{1-\frac 2 p -2\th}|\widehat{a}(\xi)|^2\,d\xi\\
&\leq&\int_{\R^3}|\widehat{a}(\xi)|^{\frac 2 p
}\bigl(|\xi||\widehat{a}(\xi)|\bigr)^{\frac 2 {p'}} |\xi_{\rm
h}|^{2\left(-3\al(r)+\th\right)}|\xi_3|^{-2\th}\,d\xi. \eeno
Applying H\"older's inequality with measure $|\xi_{\rm
h}|^{2(-3\al(r)+\th)}|\xi_3|^{-2\th}\,d\xi$ yields
\beq
\label{rq.3}
\|a\|_{\dH^{\frac 1 2   -3\al(r)+\theta, \frac 1 2 n  bsp; -\frac 1 p
-\theta}} \leq \|a\|_{\cH^{\theta,r}}^{ \frac 1 p } \|\nabla
a\|_{\cH^{\theta,r}}^{\frac 1 {p'} }.
\eeq We then infer that
$$
{
 \bigl|\bigl(Q_1(v,v) \, |\, \p_3v^3\bigr)_{\cH^{\theta,r}}\bigr| \lesssim \|v^3\|_{\dH^{\frac 1 2   +\frac 2 p }} \Bigl(
\bigl\|\, \om_{\frac r 2}\bigr\|_{L^2}^{ 2\left(\frac 1 p +2\al(r)\right)}
\bigl\|\nabla \om_{\frac r 2}\bigr\|_{L^2}
^{\frac 2 {p'}}}{{} +\|\partial_3v^3\|_{\cH^{\theta,r}}^{\frac 2 p }
\|\nabla
\partial_3v^3\|_{\cH^{\theta,r}}^{\frac 2 {p'}}\Bigr).}
$$
Convexity inequality ensures \beq \label{estimatedivhdemoeq1}
\begin{split}
 \bigl|\bigl(Q_1(v,v) \, |\, \p_3v^3\bigr)_{\cH^{\theta,r}}\bigr| &\leq \frac 1 {6} \|\nabla \partial_3v^3\|_{\cH^{\theta,r}}^2
 + C \|v^3\|_{\dH^{\frac 1 2   +\frac 2 p }}^p \|\partial_3v^3\|_{\cH^{\theta,r}}^2 \\
&\qquad\qquad\qquad{} + C\|v^3\|_{\dH^{\frac 1 2   +\frac 2 p }} \bigl\| \,
\om_{\frac r 2}\bigr\|_{L^2}^{ 2\left(\frac 1 p +2\al(r)  \right)}
\bigl\|\nabla \om_{\frac r 2}\bigr\|_{L^2} ^{\frac 2 {p'}}.
\end{split}
\eeq

\medbreak In order to
estimate~$\bigl(Q_2(v,v)\,|\,\partial_3v^3\bigr)_{\cH^{\theta,r}}$,
we first make the following observation: since $\th>3\al(r)-2/p$ and
$4<p<\f{2r}{2-r},$ we have
$$
 \frac 2p  + 3\alpha(r) -1\leq \th<5\al(r)<\f2p+3\al(r),
 $$
 and hence
 $$
 |\xi_{\rm h}|^{2\left(1-\frac 2 p -6\al(r)+2\th\right)}|\xi_3|^{2\left(\frac 2 p +3\al(r)-2\th\right)}
 \leq
 |\xi_{\rm h}|^{2\left(-3\al(r)+\th\right)}|\xi_3|^{-2\th}|\xi|^2.
 $$
We infer that
 \ben
\|a\|_{\dH^{1-6\al(r)-\frac 2 p +2\th,\frac 2 p +3\al(r)-2\th}}^2&=&\int_{\R^3}|\xi_{\rm
h}|^{2\left(1-\frac 2 p -6\al(r)+2\th\right)}|\xi_3|^{2\left(\frac 2 p +3\al(r)-2\th\right)}|\widehat{a}(\xi)|^2\,d\xi\nonumber\\
&\leq& \int_{\R^3}|\xi_{\rm
h}|^{2\left(-3\al(r)+\th\right)}|\xi_3|^{-2\th}\bigl(|\xi||\widehat{a}(\xi)|\bigr)^2\,d\xi\label{qw.5}\\
&=&\|\na a\|_{\cH^  {\theta,r}}^2. \nonumber \een Along the same lines,
one has \beq \label{pou}
\|a\|_{\dH^{1-6\al(r)+2\th,3\al(r)-2\th}}^2\leq \|\na
a\|_{\cH^{\theta,r}}^2.\eeq While we get by applying Bony's
decomposition  \eqref{pd} in the vertical variable for
$\p_3v^\ell\p_\ell v^3$ that \beno \p_3v^\ell\p_\ell v^3=T^{\rm
v}(\p_3v^\ell,\p_\ell v^3)+\bar{T}^{\rm v}(\p_3v^\ell,\p_\ell
v^3)+R^{\rm v}(\p_3v^\ell,\p_\ell v^3). \eeno The law of product of
Lemma\refer{lem2.2qw} implies that \beno \bigl\|T^{\rm
v}(\p_3v^\ell,\p_\ell v^3)&+&\bar{T}^{\rm v}(\p_3v^\ell,\p_\ell
v^3)\bigr\|_{\dH^{\frac 2 p -1,-3\al(r)-\frac 2 p }}\\
 & \lesssim &
\|\p_3v^\ell\|_{\bigl(\dB^1_{2,1}\bigr)_{\rm h}\bigl(
\dB^{-3\al(r)-\frac 2 p }_{2,1}\bigr)_{\rm v}}\|\partial_\ell v^3\|_{\bigl(\dH^{\frac 2 p -1}\bigr)_{\rm h}\bigl(\dB^{\frac 1 2   }_{2,1}\bigr)_{\rm v}}\\
& \lesssim & \|v^\ell\|_{\bigl(\dB^{1}_{2,1}\bigr)_{\rm
h}\bigl (\dB^{1-3\al(r)-\frac 2 p }_{2,1}\bigr)_{\rm v}}\|v^3\|_{\dH^{
\frac 1 2   +\frac 2 p }} . \eeno As we have~$\|\partial_\ell
v^3\|_{\bigl(\dH^{\frac 2 p -1}\bigr)_{\rm
h}\bigl(\dB^{\frac 1 2   }_{2,1}\bigr)_{\rm v}} \lesssim
\|v^3\|_{\bigl(\dH^{\frac 2 p }\bigr)_{\rm
h}\bigl(\dB^{\frac 1 2   }_{2,1}\bigr)_{\rm v}}\leq \|v^3\|_{\dH^{
\frac 1 2   +\frac 2 p }}.$

Thus we get, by applying \eqref{estimdivhlemme1demoeq1} that \beno
\begin{split}
\Bigl|\Bigl(\bigl(\Id+2&\partial_3^2\D^{-1} \bigr)\sum_{\ell=1}^2
\bigl(T^{\rm v}(\p_3v^\ell,\p_\ell v^3)+\bar{T}^{\rm
v}(\p_3v^\ell,\p_\ell
v^3)\bigr)\,|\,\partial_3v^3\Bigr)_{\cH^{\theta,r}}\Bigr|\\
\lesssim &\bigl\|T^{\rm v}(\p_3v^\ell,\p_\ell v^3)+\bar{T}^{\rm
v}(\p_3v^\ell,\p_\ell v^3)\bigr\|_{\dH^{\frac 2 p -1,-3\al(r)-\frac
2 p }}\|\p_3v^3\|_{\dH^{1-6\al(r)-\f2p+2\th,\f2p+3\al(r)-2\th}}\\
\lesssim & \|v^3\|_{\dH^{\frac 1 2   +\frac 2 p }}\|v^{\rm
h}\|_{\bigl(\dB^{1}_{2,1}\bigr)_{\rm h}\bigl(\dB^{1-3\al(r)-\frac 2
p }_{2,1}\bigr)_{\rm v}}\|\na\p_3v^3\|_{\cH^{\theta,r}}.\end{split}
\eeno

Whereas applying the law of product of Lemma\refer{lem2.2qw} once
again yields \beno \|R^{\rm v}(\p_3v^\ell,\p_\ell
v^3)\|_{\dH^{-6\al(r)+2\th,\frac12-\frac2p-2\th}}\lesssim
\|\p_3v^\ell\|_{\bigl(\dB^1_{2,1}\bigr)_{\rm
h}\bigl(\dB^{-3\al(r)-\frac2p}_{2,1}\bigr)_{\rm v}}\|\p_\ell
v^3\|_{\bigl(\dB^{-6\al(r)+2\th}_{2,2}\bigr)_{\rm
h}\bigl(\dB^{1+3\al(r)-2\th}_{2,2}\bigr)_{\rm v}}, \eeno which
together with \eqref{estimdivhlemme1demoeq2} and \eqref{pou} ensures
\beno \begin{split} \Bigl|\Bigl(\bigl(\Id+2\partial_3^2\D^{-1}
\bigr)&\sum_{\ell=1}^2 R^{\rm v}(\p_3v^\ell,\p_\ell
v^3)\,|\,\partial_3v^3\Bigr)_{\cH^{\theta,r}}\Bigr|\\
\lesssim & \|R^{\rm v}(\p_3v^\ell,\p_\ell
v^3)\|_{\dH^{-6\al(r)+2\th,\frac12-\frac2p-2\th}}\|\p_3v^3\|_{\dH^{0,-\f12+\f2p}}\\
\lesssim &\|v^3\|_{\dH^{\frac 1 2   +\frac 2 p }}\|v^{\rm
h}\|_{\bigl(\dB^{1}_{2,1}\bigr)_{\rm h}\bigl(\dB^{1-3\al(r)-\frac 2
p }_{2,1}\bigr)_{\rm v}}\|\na\p_3v^3  \|_{\cH^{\theta,r}}.\end{split}
\eeno

Therefore, by virtue of Proposition\refer{BiotSavartBesovaniso}, we
infer that
\[
\begin{split}
\bigl|\bigl(Q_2(v,v)\,|\,\partial_3v^3\bigr)_{\cH^{\theta,r}}\bigr|
&\lesssim
 \|v^3\|_{\dH^{\frac 1 2   +\frac 2 p }}\\
 &\ \ \ {}\times \Bigl(
\bigl\|\, \om_{\frac r 2}\bigr\|_{L^2}^{2\left(\al(r) +\frac 1 p \right) }
\bigl\|\nabla \om_{\frac r 2}\bigr\|_{L^2} ^{1-\frac 2 p }
+\|\partial_3v^3\|_{\cH^{\theta,r}}^{\frac 2 p } \|\nabla
\partial_3v^3\|_{\cH^{\theta,r}}^{1-\frac 2 p }\Bigr)
 \|\nabla \partial_3v^3\|_{\cH^{\theta,r}}.
\end{split}
\] Applying convexity inequality yields
\beq
\label{estimatedivhdemoeq2}
\begin{split}
\bigl|\bigl(Q_2(v,v)\,|\,\partial_3v^3\bigr)_{\cH^{\theta,r}}\bigr|\,\leq\,
&\frac 1 {6} \|\nabla \partial_3v^3\|_{\cH^{\theta,r}}^2
 + C \|v^3\|_{\dH^{\frac 1 2   +\frac 2 p }}^p \|\partial_3v^3\|_{\cH^{\theta,r}}^2 \\
&\qquad\qquad{} +  C\|v^3\|_{\dH^{\frac 1 2   +\frac 2 p }}^2 \bigl\|\, \om_{\frac r
2}\bigr\|_{L^2}^{ 4\left(\al(r)+\frac 1 p \right) } \bigl\|\nabla
\om_{\frac34}\bigr\|_{L^2} ^{2\left(1-\frac 2 p \right)}.
\end{split}
\eeq

{} \medbreak Finally let us estimate~$\bigl(Q_3(v,v) \, |\,
\p_3v^3\bigr)_{\cH^{\theta,r}}$.

\begin{lem}
\label{estimatedivhdemoeq3}
We have the following inequality.
$$
\longformule{
\bigl|\bigl( v^{\rm h}\cdot \nabla_{\rm h} \partial_3
v^3\,|\,\partial_3v^3\bigr)_{\cH^{\theta,r}}\bigr|
 \lesssim \|v^3\|_{\dH^{\frac 1 2   +\frac 2 p }}
 \Bigl( \|\na_{\rm h}v^{\rm
h}\|_{\dH^{\frac 1 2   -3\al(r)+\th,\frac 1 2   -\frac 1 p -\th}}^2
}
{
+\|\p_3v^3\|_{\dH^{\frac 1 2   -3\al(r)+\th,\frac 1 2   -\frac 1 p -\th}}^2+\|v^{\rm
h}\|_{\bigl(\dB^{1}_{2,1}\bigr)_{\rm
h}\bigl(\dB^{1-3\al(r)-\frac 2 p }_{2,1}\bigr)_{\rm
v}}\|\na\p_3v^3\|_{\cH^{\theta,r}}\Bigr).
}
$$
\end{lem}
\begin{proof}  Let us  use the
following alternative definition for the inner-product in
$\cH^{\theta,r}:$ based on the fact that the space~$\cH^{\theta,r}$
is equal to the space~$\bigl(\dB^{-3\al(r)+\th}_{2,2}\bigr)_{\rm
h}\bigl(\dB^{-\th}_{2,2}\bigr)_{\rm v}$ of
De  finition\refer{anibesov}. \beq \label{rq.5} \bigl(v^{\rm
h}\cdot\na_{\rm h}\p_3v^3\ |\
\p_3v^3\bigr)_{\cH^{\theta,r}}=\sum_{k,\ell\in\Z^2}2^{2k\left(-3\al(r)+\th\right)}2^{-2\ell\th}\bigl(\D_k^{\rm
h}\D_{\ell}^{\rm v}(v^{\rm h}\cdot\na_{\rm h}\p_3v^3)\ |\ \D_k^{\rm
h}\D_{\ell}^{\rm v}\p_3v^3\bigr)_{L^2}. \eeq

By using Bony's decomposition \eqref{pd} to $v^{\rm h}\cdot\na_{\rm
h}\p_3v^3$ for both horizontal and vertical variables, we write that
\beq\label{rq.6}
\begin{split}
v^{\rm h}\cdot\na_{\rm h}\p_3v^3=&\bigl(T^{\rm h}+R^{\rm
h}+\bar{T}^{\rm h}\bigr)\bigl(T^{\rm v}+R^{\rm v}+\bar{T}^{\rm
v}\bigr)(v^{\rm h},\na_{\rm h}\p_3v^3)\\
=&T^{\rm h}T^{\rm v}(v^{\rm h},\na_{\rm h}\p_3v^3)+A+B \with\\
A\eqdefa & T^{\rm h}R^{\rm v}(v^{\rm h},\na_{\rm h}\p_3v^3)+T^{\rm
h}\bar{T}^{\rm v}(v^{\rm
h},\na_{\rm h}\p_3v^3)\\
B\eqdefa&R^{\rm h}T^{\rm v}(v^{\rm h},\na_{\rm h}\p_3v^3)+R^{\rm
h}R^{\rm v}(v^{\rm h},\na_{\rm h}\p_3v^3)+R^{\rm h}\bar{T}^{\rm
v}(v^{\rm
h},\na_{\rm h}\p_3v^3)\\
&{}+\bar{T}^{\rm h}{T}^{\rm v}(v^{\rm h},\na_{\rm
h}\p_3v^3)+\bar{T}^{\rm h}R^{\rm v}(v^{\rm h},\na_{\rm
h}\p_3v^3)+\bar{T}^{\rm h}\bar{T}^{\rm v}(v^{\rm h},\na_{\rm
h}\p_3v^3). \end{split} \eeq

 \no{\Large $\bullet$} The estimate of $\bigl(\D_k^{\rm h}\D_\ell^{\rm v}T^{\rm
h}T^{\rm v}(v^{\rm h},\na_{\rm h}\p_3v^3)\ |\ \D_k^{\rm
h}\D_\ell^{\rm v}\p_3v^3\bigr)_{L^2}.$

By applying  commutator's argument and also considering the support
to the Fourier transform of the terms in $T^{\rm h}T^{\rm v}(v^{\rm
h},\na_{\rm h}\p_3v^3),$ we write
\beno
\begin{split}
\bigl(\D_k^{\rm h}\D_\ell^{\rm v}&T^{\rm h}T^{\rm v}(v^{\rm
h},\na_{\rm h}\p_3v^3)\ |\ \D_k^{\rm h}\D_\ell^{\rm v}\p_3v^3\bigr)_{L^2}
\eqdefa I_{k,\ell}^1+I_{k,\ell}^2+I_{k,\ell}^3\with \\
 I_{k,\ell}^1 &\eqdefa
 \sum_{\substack{|k'-k|\leq 4\\|\ell'-\ell|\leq
4}}\Bigl(\bigl[\D_k^{\rm h}\D_\ell^{\rm v}, S_{k'-1}^{\rm
h}S_{\ell'-1}^{\rm v}v^{\rm h}\bigr]\D_{k'}^{\rm h}\D_{\ell'}^{\rm
v}\na_{\rm h}\p_3v^3\ \big|\ \D_k^{\rm h}\D_\ell^{\rm v}\p_3v^3\Bigr)_{L^2},\\
I_{k,\ell}^2&\eqdefa  \sum_{\substack{|k'-k|\leq 4\\|\ell'-\ell|\leq
4}}
\Bigl( \bigl(S_{k'-1}^{\rm h}S_{\ell'-1}^{\rm v}v^{\rm
h}-S_{k-1}^{\rm h}S_{\ell-1}^{\rm v}v^{\rm h}\bigr)
\D_{k'}^{\rm h}\D_{\ell'}^{\rm v}\D_{k}^{\rm h}\D_{\ell}^{\rm v}\  na_{\rm
h}\p_3v^3\,\big |\, \D_k^{\rm h}\D_\ell^{\rm
v}\p_3v^3\Bigr)_{L^2}\ \ \andf \\
 I_{k,\ell}^3 &\eqdefa-\f12\bigl( S_{k-1}^{\rm h}S_{\ell-1}^{\rm v}\dive_{\rm h}v^{\rm
h}\D_{k}^{\rm h}\D_{\ell}^{\rm v}\p_3v^3\ |\ \D_k^{\rm
h}\D_\ell^{\rm v}\p_3v^3\bigr)_{L^2}.
\end{split}
\eeno It follows from  a standard commutator's estimate  (see for instance\ccite{BCD}) that
$$\longformule{ \bigl|I_{k,\ell}^1\bigr|\lesssim
\sum_{\substack{|k'-k|\leq 4\\|\ell'-\ell|\leq 4}}\Bigl(2^{-k}\|
S_{k'-1}^{\rm h}S_{\ell'-1}^{\rm v}\na_{\rm h}v^{\rm
h}\|_{L^\infty}}{{}+2^{-\ell}\| S_{k'-1}^{\rm h}S_{\ell'-1}^{\rm
v}\p_3v^{\rm h}\|_{L^\infty}\Bigr)\|\D_{k'}^{\rm h}\D_{\ell'}^{\rm
v}\na_{\rm h}\p_3v^3\|_{L^2}\|\D_k^{\rm h}\D_\ell^{\rm
v}\p_3v^3\|_{L^2}.}$$ Note that applying Lemma \ref{lemBern} gives
\beno \| S_{k'-1}^{\rm h}S_{\ell'-1}^{\rm v}\na_{\rm h}v^{\rm
h}\|_{L^\infty}\lesssim
2^{k'\left(\frac 1 2   +3\al(r  )-\th\right)}2^{\ell'\left(\frac 1 p +\th\right)}\|\na_{\rm
h}v^{\rm h}\|_{\dH^{\frac 1 2   -3\al(r)+\th,\frac 1 2   -\frac 1 p -\th}}, \eeno
from which, we infer that
\beno
\begin{split}
&2^{-k}\sum_{\substack{|k'-k|\leq 4\\|\ell'-\ell|\leq 4}}\|
S_{k'-1}^{\rm h}S_{\ell'-1}^{\rm v}\na_{\rm h}v^{\rm
h}\|_{L^\infty}\|\D_{k'}^{\rm h}\D_{\ell'}^{\rm v}\na_{\rm
h}\p_3v^3\|_{L^2}\|\D_k^{\rm h}\D_\ell^{\rm
v}\p_3v^3\|_{L^2}\\
&\quad\lesssim\sum_{\substack{|k'-k|\leq 4\\|\ell'-\ell|\leq
4}}c_{k',\ell'}2^{2k'\left(3\al(r)-\th\right)}2^{\ell'\left(-\frac 1 2   +\frac 2 p +2\th\right)}\|\na_{\rm
h}v^{\rm
h}\|_{\dH^{\frac 1 2   -3\al(r)+\th,\frac 1 2   -\frac 1 p -\th}}\\
&\qquad\qquad\qquad\qquad\qquad\times\|\p_3v^3\|_{\dH^{\frac 1 2   -3\al(r)+\th,\frac 1 2   -\frac 1 p -\th}}c_{k,\ell}2^{\ell\left(\frac 1 2   -\frac 2 p \right)}\|\p_3v^3\|_{\dH^{0,-\frac 1 2    +\frac 2 p }}\\
&\quad\lesssim
d_{k,\ell}2^{2k\left(3\al(r)-\th\right)}2^{2\ell\th}\|v^3\|_{\dH^{\frac 1 2   +\frac 2 p }}\Bigl(\|\na_{\rm
h}v^{\rm
h}\|_{\dH^{\frac 1 2   -3\al(r)+\th,\frac 1 2   -\frac 1 p -\th}}^2+
\|\p_3v^3\|_{\dH^{\frac 1 2   -3\al(r)+\th,\frac 1 2   -\frac 1 p -\th}}^2\Bigr).
\end{split}
\eeno
Here and in what follows, we always denote
$\bigl(c_{k,\ell}\bigr)_{k,\ell\in\Z^2}$ (resp.
$\bigl(d_{k,\ell}\bigr)_{k,\ell\in\Z^2}$) to be a generic element of
the sphere in~$\ell^2(\Z^2)$ (resp. $\ell^1(\Z^2)$). The same
estimate holds for $I_{k,\ell}^3.$

Likewise, since \beno \| S_{k'-1}^{\rm h}S_{\ell'-1}^{\rm
v}\p_3v^{\rm h}\|_{L^\infty}\lesssim
2^{\ell'\left(\frac 1 2   +3\al(r)+\frac 2 p \right)}\|v^{\rm
h}\|_{\bigl(\dB^1_{2,1}\bigr)_{\rm
h}\bigl(\dB^{1-3\al(r)-\frac 2 p }_{2,1}\bigr)_{\rm v}}, \eeno and
$\|\na_{\rm h}v^3\|_{\bigl(\dH^{-1+\frac 2 p }\bigr)_{\rm
h}\bigl(\dB^{\frac 1 2   }_{2,1}\bigr)_{\rm v}}\lesssim
\|v^3\|_{\bigl(\dH^{\frac 2 p }\bigr)_{\rm
h}\bigl(\dB^{\frac 1 2   }_{2,1}\bigr)_{\rm v}}\lesssim
\|v^3\|_{\dH^{\frac 1 2   +\frac 2 p }},$  we have \beno
\begin{split}
&2^{-\ell}\sum_{\substack{|k'-k|\leq 4\\|\ell'-\ell|\leq 4}}\|
S_{k'-1}^{\rm h}S_{\ell'-1}^{\rm v}\p_3v^{\rm
h}\|_{L^\infty}\|\D_{k'}^{\rm h}\D_{\ell'}^{\rm v}\na_{\rm
h}\p_3v^3\|_{L^2}\|\D_k^{\rm h}\D_\ell^{\rm v}\p_3v^3\|_{L^2}\\
&\quad\lesssim 2^{-\ell} \sum_{\substack{|k'-k|\leq
4\\|\ell'-\ell|\leq
4}}c_{k',\ell'}2^{k'\left(1-\frac 2  p \right)}2^{\ell'\left(1+3\al(r)+\frac 2 p \right)}\|v^{\rm
h}\|_{\bigl(\dB^1_{2,1}\bigr)_{\rm
h}\bigl(\dB^{1-3\al(r)-\frac 2 p }_{2,1}\bigr)_{\rm v}}\|\na_{\rm
h}v^3\|_{\bigl(\dH^{-1+\frac 2 p }\bigr)_{\rm
h}\bigl(\dB^{\frac 1 2   }_{2,1}\bigr)_{\rm v}}\\
&\qquad\qquad\qquad\times
c_{k,\ell}2^{-k\left(1-\frac 2 p -6\al(r)+2\th\right)}2^{-\ell\left(\frac 2 p +3\al(r)-2\th\right)}\|\p_3v^3\|_{\dH^{1-\frac 2 p -6\al(r)+2\th,\frac 2 p +3\al(r)-2\th}}\\
&\quad\lesssim d_{k,\ell}
2^{2k\left(3\al(r)-\th\right)}2^{2\ell\th}\|v^3\|_{\dH^{\frac 1 2   +\frac 2 p }}\|v^{\rm
h}\|_{\bigl(\dB^1_{2,1}\bigr)_{\rm
h}\bigl(\dB^{1-3\al(r)-\frac 2 p }_{2,1}\bigr)_{\rm
v}}\|\p_3v^3\|_{\dH^{1-\frac 2 p -6\al(r)+2\th,\frac 2 p +3\al(r)-2\th}}.
\end{split}
\eeno

Therefore, by virtue of \eqref{qw.5}, we obtain \beq\label{rq.7}
\begin{split}
\bigl|I_{k,\ell}^1\bigr|\lesssim &
d_{k,\ell}2^{2k\left(3\al(r)-\th\right)}2^{2\ell\th}\|v^3\|_{\dH^{\frac 1 2   +\frac 2 p }}\Bigl(\|\na_{\rm
h}v^{\rm h}\|_{\dH^{\frac 1 2   -3\al(r)+\th,\frac 1 2   -\frac 1 p -\th}}^2\\
&{}+ \|\p_3v^3\|_{\dH^{\frac 1 2   -3\al(r)+\th,\frac 1 2   -\frac 1 p -\th}}^2+\|v^{\rm
h}\|_{\bigl(\dB^1_{2,1}\bigr)_{\rm
h}\bigl(\dB^{1-3\al(r)-\frac 2 p }_{2,1}\bigr)_{\rm
v}}\|\na\p_3v^3\|_{\cH^{\theta,r}}\Bigr).
\end{split}
\eeq The same argument gives the same estimate for $I_{k,\ell}^2.$
We thus conclude that
$$
\bigl(\D_k^{\rm h}\D_\ell^{\rm v}T^{\rm
h}T^{\rm v}(v^{\rm h},\na_{\rm h}\p_3v^3)\ |\ \D_k^{\rm
h}\D_\ell^{\rm v}\p_3v^3\bigr)_{L^2}
$$
 also verifies the Estimate
\eqref{rq.7}.\\

 \no{\Large $\bullet$} The estimate of $\bigl(\D_k^{\rm h}\D_\ell^{\rm v}A\ |\ \D_k^{\rm
h}\D_\ell^{\rm v}\bigr)_{L^2}.$

We first get, by applying Lemma \ref{lemBern}, that \beq\label{qw.8}
\begin{split}
&\|S_{k'-1}^{\rm h}\D_{\ell'}^{\rm v}v^{\rm h}\|_{L^\infty_{\rm
h}(L^2_{\rm v})}\lesssim
2^{-\ell'\left(1-3\al(r)-\frac 2 p \right)}\|v^{\rm
h}\|_{\bigl(\dB^1_{2,1}\bigr)_{\rm
h}\bigl(\dB^{1-3\al(r)-\frac 2 p }_{2,1}\bigr)_{\rm v}},\\
&\|\D_{k'}^{\rm h}{S}_{\ell'-1}^{\rm v}\na_{\rm
h}\p_3v^3\|_{L^2_{\rm h}(L^\infty_{\rm v})}\lesssim
c_{k',\ell'}2^{k'\left(1-\frac 2 p \right)}2^{\ell'}\|v^3\|_{\bigl(\dH^{\frac 2 p }\bigr)_{\rm
h}\bigl(\dB^{\frac 1 2   }_{2,1}\bigr)_{\rm v}}.
\end{split}
\eeq In view of \eqref{qw.8}, Lemma \ref{lemBern}, and also
considering the support to the Fourier transform to terms in $
T^{\rm h}R^{\rm v}(v^{\rm h},\na_{\rm h}\p_3v^3),$  we write \beno
\begin{split}
\|&\D_k^{\rm h}\D_\ell^{\rm v} T^{\rm h}R^{\rm v}(v^{\rm h},\na_{\rm
h}\p_3v^3)\|_{L^2} \lesssim {}2^{\ell/2}\sum_{\substack{|k'-k  |\leq
4\\\ell'\geq\ell-3}}\|S_{k'-1}^{\rm h}\D_{\ell'}^{\rm v}v^{\rm
h}\|_{L^\infty_{\rm h}(L^2_{\rm v})}\|\D_{k'}^{\rm
h}\wt{\D}_{\ell'}^{\rm v}\na_{\rm h}\p_3v^3\|_{L^2}\\
&\qquad\qquad\qquad\lesssim {}2^{\ell/2}\sum_{\substack{|k'-k|\leq
4\\\ell'\geq\ell-3}}c_{k',\ell'}2^{2k'\left(\frac 1 p +3\al(r)-\th\right)}2^{-\ell'(1-2\th)}
\|v^{\rm h}\|_{\bigl(\dB^1_{2,1}\bigr)_{\rm
h}\bigl(\dB^{1-3\al(r)-\frac 2 p }_{2,1}\bigr)_{\rm
v}}\\
&\qquad\qquad\qquad\qquad\qquad\qquad\qquad\qquad\qquad\qquad\qquad\qquad\times\|\p_3v^3\|_{\dH^{1-\frac 2 p -6\al(r)+2\th,\frac 2 p +3\al(r)-2\th}},
\end{split}
\eeno which gives $$\longformule{\|\D_k^{\rm h}\D_\ell^{\rm v}
T^{\rm h}R^{\rm v}(v^{\rm h},\na_{\rm h}\p_3v^3)\|_{L^2} \lesssim
c_{k,\ell}2^{2k\left(\frac 1 p
+3\al(r)-\th\right)}2^{-\ell\left(\frac 1 2
-2\th\right)}}{{}\times\|v^{\rm h}\|_{\bigl(\dB^1_{2,1}\bigr)_{\rm
h}\bigl(\dB^{1-3\al(r)-\frac 2 p }_{2,1}\bigr)_{\rm
v}}\|\p_  3v^3\|_{\dH^{1-\frac 2 p -6\al(r)+2\th,\frac 2 p
+3\al(r)-2\th}}. }$$ Therefore since \beno \|\D_k^{\rm
h}\D_\ell^{\rm v}\p_3v^3\|_{L^2}&\lesssim&
c_{k,\ell}2^{-\f{2k}p}2^{\f{\ell}2}\|v^3\|_{\bigl(\dH^{\frac 2 p
}\bigr)_{\rm
h}\bigl(\dB^{\frac 1 2   }_{2,1}\bigr)_{\rm v}}\\
&\lesssim& c_{k,\ell}2^{-\f{2k}p}2^{\f\ell2}\|v^3\|_{\dH^{\frac 1 2
+\frac 2 p }}, \eeno we obtain \beq\label{qw.7}
\begin{split}
|\bigl(\D_k^{\rm h}\D_\ell^{\rm v}&T^{\rm h}R^{\rm v}(v^{\rm
h},\na_{\rm h}\p_3v^3)\ |\ \D_k^{\rm h}\D_\ell^{\rm
v}\p_3v^3\bigr)_{L^2}\bigr| \lesssim
d_{k,\ell}2^{2k\left(3\al(r)-\th\right)}2^{2\ell\th}\|v^3\|_{\dH^{\frac 1 2   +\frac 2 p }}\\
 &\qquad\qquad\times\|v^{\rm
h}\|_{\bigl(\dB^1_{2,1}\bigr)_{\rm
h}\bigl(\dB^{1-3\al(r)-\frac 2 p }_{2,1}\bigr)_{\rm
v}}\|\p_3v^3\|_{\dH^{1-\frac 2 p -6\al(r)+2\th,\frac 2 p +3\al(r)-2\th}}.
\end{split}
\eeq

 Along the same lines, we infer from \eqref{qw.8} that \beno
\begin{split}
\|\D_k^{\rm h}\D_\ell^{\rm v} T^{\rm h}\bar{T}^{\rm v}(v^{\rm
h},\na_{\rm h}\p_3v^3)\|_{L^2} \lesssim&
{}\sum_{\substack{|k'-k|\leq 4\\|\ell'-\ell|\leq 4}}\|S_{k'-1}^{\rm
h}\D_{\ell'}^{\rm v}v^{\rm h}\|_{L^\infty_{\rm h}(L^2_{\rm
v})}\|\D_{k'}^{\rm
h}{S}_{\ell'-1}^{\rm v}\na_{\rm h}\p_3v^3\|_{L^2_{\rm h}(L^\infty_{\rm v})}\\
\lesssim&
c_{k,\ell}2^{k\left(1-\frac 2 p \right)}2^{\ell\left(3\al(r)+\frac 2 p \right)}\\
&\qquad{}\times\|v^{\rm h}\|_{\bigl(\dB^1_{2,1}\bigr)_{\rm
h}\bigl(\dB^{1-3\al(r)-\frac 2 p }_{2,1}\bigr)_{\rm
v}}\|v^3\|_{\bigl(\dH^{\frac 2 p }\bigr)_{\rm
h}\bigl(\dB^{\frac 1 2   }_{2,1}\bigr)_{\rm v}},
\end{split}
\eeno so that we obtain \beno
\begin{split}
|\bigl(\D_k^{\rm h}\D_\ell^{\rm v}&T^{\rm h}\bar{T}^{\rm v}(v^{\rm
h},\na_{\rm h}\p_3v^3)\ |\ \D_k^{\rm h}\D_\ell^{\rm
v}\p_3v^3\bigr)_{L^2  }\bigr| \lesssim
d_{k,\ell}2^{2k\left(3\al(r)-\th\right)}2^{2\ell\th}\|v^3\|_{\dH^{\frac 1 2   +\frac 2 p }}\\
 &\qquad\qquad\times\|v^{\rm
h}\|_{\bigl(\dB^1_{2,1}\bigr)_{\rm
h}\bigl(\dB^{1-3\al(r)-\frac 2 p }_{2,1}\bigr)_{\rm
v}}\|\p_3v^3\|_{\dH^{1-\frac 2 p -6\al(r)+2\th,\frac 2 p +3\al(r)-2\th}},
\end{split}
\eeno from which, \eqref{qw.5} and \eqref{qw.7}, we deduce that
\beq\label{rq.8}
\begin{split}
|\bigl(\D_k^{\rm h}\D_\ell^{\rm v}&A\ |\ \D_k^{\rm h}\D_\ell^{\rm
v}\p_3v^3\bigr)_{L^2}\bigr| \lesssim
d_{k,\ell}2^{2k\left(3\al(r)-\th\right)}2^{2\ell\th}\|v^3\|_{\dH^{\frac 1 2   +\frac 2 p }}\\
 &\qquad\qquad\qquad\qquad\qquad\times\|v^{\rm
h}\|_{\bigl(\dB^1_{2,1}\bigr)_{\rm
h}\bigl(\dB^{1-3\al(r)-\frac 2 p }_{2,1}\bigr)_{\rm
v}}\|\na\p_3v^3\|_{\cH^{\theta,r}}.
\end{split}
\eeq

\no{\Large $\bullet$} The estimate of $\bigl(\D_k^{\rm
h}\D_\ell^{\rm v}B\ |\ \D_k^{\rm h}\D_\ell^{\rm v}\bigr)_{L^2}.$

Again considering the support to the Fourier transform to terms in $
R^{\rm h}R^{\rm v}(v^{\rm h},\na_{\rm h}\p_3v^3),$  we get, by
applying Lemma \ref{lemBern}, that \beno
\begin{split}
\|&\D_k^{\rm h}\D_\ell^{\rm v} R^{\rm h}R^{\rm v}(v^{\rm h},\na_{\rm
h}\p_3v^3)\|_{L^2} \lesssim {}2^k2^{\ell/2}\sum_{\substack{k'\geq k-
3\\\ell'\geq\ell-3}}\|\D_{k'}^{\rm h}\D_{\ell'}^{\rm v}v^{\rm
h}\|_{L^2}\|\wt{\D}_{k'}^{\rm
h}\wt{\D}_{\ell'}^{\rm v}\na_{\rm h}\p_3v^3\|_{L^2}\\
&\lesssim 2^k2^{\ell/2}\sum_{\substack{k'\geq k-
3\\\ell'\geq\ell-3}}c_{k',\ell'}2^{-k'\left(1-6\al(r)+2\th\right)}2^{-\ell'\left(1-\frac 2 p -2\th\right)}\\
&\qquad\qquad\qquad\qquad\times\|\na_{\rm h}v^{\rm
h}\|_{\dH^{\frac 1 2   -3\al(r)+\th,\frac 1 2   -\frac 1 p -\th}}\|\p_3v^3\|_{\dH^{\frac 1 2   -3\al(r)+\th,\frac 1 2   -\frac 1 p -\th}}\\
&\lesssim
c_{k,\ell}2^{2k(3\al(r)-\th)}2^{-\ell\left(\frac 1 2  nbsp;  -\frac 2 p -2\th\right)}\|\na_{\rm
h}v^{\rm
h}\|_{\dH^{\frac 1 2   -3\al(r)+\th,\frac 1 2   -\frac 1 p -\th}}\|\p_3v^3\|_{\dH^{\frac 1 2   -3\al(r)+\th,\frac 1 2   -\frac 1 p -\th}},
\end{split}
\eeno by using the fact that $3\al(r)-\frac 1 2   <0<\th<\frac 1 2   -\frac 2 p .$

Likewise, we have \beno
\begin{split}
\|&\D_k^{\rm h}\D_\ell^{\rm v} R^{\rm h}T^{\rm v}(v^{\rm h},\na_{\rm
h}\p_3v^3)\|_{L^2} \lesssim {}2^k\sum_{\substack{k'\geq k-
3\\|\ell'-\ell|\leq 4}}\|\D_{k'}^{\rm h}S_{\ell'-1}^{\rm v}v^{\rm
h}\|_{L^2_{\rm h}(L^\infty_{\rm v})}\|\wt{\D}_{k'}^{\rm
h}{\D}_{\ell'}^{\rm v}\na_{\rm h}\p_3v^3\|_{L^2}\\
&\lesssim 2^k\sum_{\substack{k'\geq k- 3\\|\ell'-\ell|\leq 4}}
c_{k',\ell'}2^{-k'\left(1-6\al(r)+2\th\right)}2^{-\ell'\left(\frac 1 2   -\frac 2 p -2\th\right)}\\
&\qquad\qquad\qquad\qquad\times \|\na_{\rm h}v^{\rm
h}\|_{\dH^{\frac 1 2   -3\al(r)+\th,\frac 1 2   -\frac 1 p -\th}}\|\p_3v^3\|_{\dH^{\frac 1 2   -3\al(r)+\th,\frac 1 2   -\frac 1 p -\th}}\\
&\lesssim
c_{k,\ell}2^{2k(3\al(r)-\th)}2^{-\ell\left(\frac 1 2   -\frac 2 p -2\th\right)}\|\na_{\rm
h}v^{\rm
h}\|_{\dH^{\frac 1 2   -3\al(r)+\th,\frac 1 2     -\frac 1 p -\th}}\|\p_3v^3\|_{\dH^{\frac 1 2   -3\al(r)+\th,\frac 1 2   -\frac 1 p -\th}}.
\end{split}
\eeno It is easy to check that all the remaining terms in $B$ given
by \eqref{rq.6} share the same estimate. Therefore, we obtain
\beq\label{rq.9}\begin{split} \bigl|\bigl(\D_k^{\rm h}\D_\ell^{\rm
v}B\ |\ \D_k^{\rm h}\D_\ell^{\rm v}\bigr)_{L^2}\bigr|\lesssim&
\|\D_k^{\rm
h}\D_\ell^{\rm v}B\|_{L^2}\|\D_k^{\rm h}\D_\ell^{\rm v}\|_{L^2}\\
\lesssim& d_{k,\ell}2^{2k(3\al(r)-\th)}2^{2\ell\th}\|\na_{\rm
h}v^{\rm
h}\|_{\dH^{\frac 1 2   -3\al(r)+\th,\frac 1 2   -\frac 1 p -\th}}\\
&\qquad\times\|\p_3v^3\|_{\dH^{\frac 1 2   -3\al(r)+\th,\frac 1 2   -\frac 1 p -\th}}\|\p_3v^3\|_{\dH^{0,-\frac 1 2   +\frac 2 p }}.
\end{split} \eeq
Inserting the Estimates \eqref{rq.7}, \eqref{rq.8} and \eqref{rq.9}
in \eqref{rq.6} leads to Lemma \ref{estimatedivhdemoeq3}.\end{proof}

Thanks to Lemma \ref{estimatedivhdemoeq3}, we get, by applying
\eqref{a.1} and Proposition \ref{BiotSavartBesovaniso}, that \beno
\begin{split}
\bigl|\bigl( v^{\rm h}\cdot &\nabla_{\rm h} \partial_3
v^3\,|\,\partial_3v^3\bigr)_{\cH^{\theta,r}}\bigr| \\
\lesssim &
\|v^3\|_{H^{\frac 1 2   +\frac 2 p }}\left(\|\om\|_{\dH^{\frac 1 2   -3\al(r)+\th,\frac 1 2   -\frac 1 p -\th}}^{2}+\|\p_3v^3\|_{\dH^{\frac 1 2   -3\al(r)+\th,\frac 1 2   -\frac 1 p -\th}}^{2}\right.\\
&\qquad\left.+\Bigl(\bigl\|\,\om_{\f{r}2}\bigr\|_{L^2}^{2\left(\al(r)+\frac 1 p \right)}
\bigl\|\nabla \om_{\frac r
2}\bigr\|_{L^2}^{1-\frac 2 p }+\|\partial_3v^3\|_{\cH^{\theta,r}}^{\frac 2 p }
\|\nabla
\partial_3v^3\|_{\cH^{\theta,r}}^{1-\frac 2 p }\Bigr)
\|\nabla\partial_3v^3\|_{\cH^{\theta,r}}\right),
\end{split}
\eeno from which, \eqref{rq.2} and \eqref{rq.3}, we infer \beq
\label{rq.4} \begin{split} \bigl |\bigl(& v^{\rm h}\cdot \nabla_{\rm
h}
\partial_3 v^3\,|\,\partial_3v^3\bigr)_{\cH^{\theta,r}}\bigr|\lesssim
\|v^3\|_{H^{\frac 1 2   +\frac 2 p }}\left(\bigl\|\,\om_{\f{r}2}\bigr\|_{L^2}^{2\left(2\al(r)+\frac 1 p \right)}
\bigl\|\nabla \om_{\frac r 2}\bigr\|_{L^2} ^{\frac 2 {p'}}\right.\\
&\qquad\qquad\left.+\Bigl(\bigl\|\,\om_{\f{r}2}\bigr\|_{L^2}^{2\left(\al(r)+\frac 1 p \right)}
\bigl\|\nabla \om_{\frac r
2}\bigr\|_{L^2}^{1-\frac 2 p }+\|\partial_3v^3\|_{\cH^{\theta,r}}^{\frac 2 p }
\|\nabla
\partial_3v^3\|_{\cH^{\theta,r}}^{1-\frac 2 p }\Bigr)
\|\nabla\partial_3v^3\|_{\cH^{\theta,r}}\right).
\end{split} \eeq

To estimate~$\bigl( v^{3} \partial_3^2
v^3\,|\,\partial_3v^3\bigr)_{\cH^{\theta,r}}$, we write, according to
\eqref{estimdivhlemme1demoeq1}, that
$$
\bigl|(f\,|\,g)_{\cH^{\theta,r}}\bigr| \leq
\|f\|_{\dH^{-1-3\al(r)+\frac 2 p +\th,-\theta}}\|g\|_{\dH^{1-3\al(r)-\frac 2 p +\theta,-\theta}}.
$$ As $\th>3\al(r)-\f2p,$
we get, by applying law of product  of Lemma\refer{lem2.2qw}  and then
Lemma\refer{isoaniso}, that \beno \bigl|\bigl( v^{3} \partial_3^2
v^3\,|\,\partial_3v^3\bigr)_{\cH^{\theta,r}}\bigr|
& \leq & \|v^3\partial_3^2v^3\|_{\dH^{-1-3\al(r)+\frac 2 p +\th,-\theta}}\|\partial_3v^3\|_{\dH^{1-3\al(r)-\frac 2 p +\theta,-\theta}}\\
& \lesssim &
 \|v^3\|_{\bigl(\dH^{\frac 2 p }\bigr)_{\rm h}\bigl(B^{\frac 1 2   }_{2,1}\bigr)_{\rm v}}
 \| \partial^2_3v^3\|_{\cH^{\theta,r}}\|\partial_3v^3\|_{\dH^{1-3\al(r)-\frac 2 p +\theta,-\theta}}  \\
 & \lesssim & \|v^3\|_{\dH^{  \frac 1 2   +\frac 2 p }}\|
 \partial^2_3v^3\|_{\cH^{\theta,r}}\|\partial_3v^3\|_{\dH^{1-3\al(r)-\frac 2 p +\theta,-\theta}}.
 \eeno
This along with the interpolation inequality which  claims that
 \beno
 \|a\|_{\dH^{1-3\al(r)-\frac 2 p +\theta,-\theta}}^2&=&\int_{\R^3}|\xi_{\rm
 h}|^{2\left(1-\frac 2 p \right)}|\xi_{\rm
 h}|^{-6\al(r)+2\th}|\xi_3|^{-2\th}|\widehat{a}(\xi)|^2\,d\xi\\
 &\leq& \Bigl(\int_{\R^3}|\xi_{\rm
 h}|^{-6\al(r)+2\th}|\xi_3|^{-2\th}|\widehat{a}(\xi)|^2\,d\xi\Bigr)^{\frac 2 p }\\
&&\qquad\quad{}\times \Bigl(\int_{\R^3}|\xi_{\rm
 h}|^{-6\al(r)+2\th}|\xi_3|^{-2\th}|\xi_{\rm
 h}|^2|\widehat{a}(\xi)|^2\,d\xi\Bigr)^{1-\frac 2 p }\\
 &\leq& \|a\|_{\cH^{\theta,r}}^{4/p}\|\na_{\rm
 h}a\|_{\cH^{\theta,r}}^{2\left(1-\frac 2 p \right)},
 \eeno
 ensures
 \beno  \bigl|\bigl( v^{3} \partial_3^  2
v^3\,|\,\partial_3v^3\bigr)_{\cH^{\theta,r}}\bigr|\lesssim
\|v^3\|_{\dH^{\frac 1 2   +\frac 2 p }}\|\p_3v^3\|_{\cH^{\theta,r}}^{\frac 2 p }\|
 \na\partial_3v^3\|_{\cH^{\theta,r}}^{\frac 2 {p'}}.\eeno
Due to \eqref{rq.4} and convexity inequality, we thus obtain \beq
\label{estimatedivhdemoeq4p}
\begin{split}
\bigl|\bigl(Q_3(v,v)\,|\,\partial_3v^3\bigr)_{\cH^{\theta,r}}\bigr|\leq
\,&\,\frac 1 {6} \|\nabla \partial_3v^3\|_{\cH^{\theta,r}}^2
 + C \|v^3\|_{\dH^{\frac 1 2   +\frac 2 p }}^p \|\partial_3v^3\|_{\cH^{\theta,r}}^2 \\
&{} + C\|v^3\|_{\dH^{\frac 1 2   +\frac 2 p }} \bigl\|\, \om_{\frac r
2}\bigr\|_{L^2}^{ 2\left(2\al(r)+\frac 1 p \right) }
\bigl\|\nabla \om_{\frac r 2}\bigr\|_{L^2}^{\frac 2 {p'}}\\
&{} + C\|v^3\|_{\dH^{\frac 1 2   +\frac 2 p }}^2 \bigl\|\, \om_{\frac r
2}\bigr\|_{L^2}^{ 4\left(\al(r)+\frac 1 p \right) } \bigl\|\nabla
\om_{\frac r 2}\bigr\|_{L^2}^{2\left(1 -\frac 2 p \right)}.
\end{split}
\eeq

\bigbreak


Now we are in a position to complete the proof of Proposition
\ref{estimadivhaniso}.

\medbreak

\begin{proof}[Conclusion of the proof to Proposition
\ref{estimadivhaniso}] By resuming the Estimates
\eqref{estimatedivhdemoeq1}, \eqref{estimatedivhdemoeq2} and
\eqref{estimatedivhdemoeq4p} into\refeq{b.4pu}, we obtain
\beq\label{b.26}
\begin{split}
\f{d}{dt}&\|\p_3v^3(t)\|_{\cH^{\theta,r}}^2+\|\na\p_3v^3(t)\|_{\cH^{\theta,r}}^2\\
&{}\leq{} C\Bigl(\|v^3\|_{\dH^{\frac 1 2   +\frac 2 p }}\bigl\|\om_{\frac r
2}\bigr\|_{L^2}^{2\left(2\al(r)+\frac 1 p \right)}\bigl\|\na\om_{\frac r
2}\bigr\|_{L^2}^{\frac 2 {p'}}
\\
&\qquad{}+\|v^3\|_{\dH^{\frac 1 2   +\frac 2 p }}^p\|\p_3v^3\|_{\cH^{\theta,r}}^2+\|v^3\|_{\dH^{\frac 1 2   +\frac 2 p }}^2
\bigl\|\om_{\frac r
2}\bigr\|_{L^2}^{4\left(\al(r)+\frac 1 p \right)}\bigl\|\na\om_{\frac r
2}\bigr\|_{L^2}^{2\left(1-\frac 2 p \right)} \Bigr).
\end{split}
\eeq
On the other hand, Inequality\refeq{initialdataHtheta}  cla  ims that
$\|\p_3v_0^3\|_{\cH^{\theta,r}}\lesssim\|v_0\|_{\dH^{1-3\al(r)}}\lesssim
\|\Om_0\|_{L^r} $. Thus Gronwall's inequality allows to conclude
the proof of Proposition \ref{estimadivhaniso}.
\end{proof}

\section{Conclusion  of the proof  of Theorem\refer{thmain}    }
The first main step is the proof of the following proposition.
\begin{prop}
\label{SophisticatedGronwall} {\sl  Let us consider a solution $v$
of~$(NS)$ given by Theorem\refer{thmain}. For any~$p$
in~$\bigl]4,\f{2r}{2-r}\bigr[$ and $\th$ in
 $\bigl]3\al(r)  -\frac 2 p ,\al(r)\bigr[,$ a
constant~$C$ exists such that, for any~$t<T^\ast$, we have \beno
\bigl\|\om_{\frac{r}2}(t)\bigr\|_{L^2}^{2(1+2p\al(r))}+\bigl\|\na\om_{\f{r}2}\bigr\|_{L^2_t(L^2)}^{2(1+2p\al(r))}
 & \leq  & C \|\Om_0\|_{L^{r}}^{r(1+2p\al(r))}\cE(t)\andf\\
 \|\p_3v^3(t)\|_{\cH^{\theta,r}}^2
 +\|\na\p_3v^3\|_{L^2_t(\cH^{\theta,r})}^2
 & \leq & \|\Om_0\|_{L^{r}}^2 \cE(t) \with\\
 \cE(t) & \eqdefa &  \exp\biggl(C\exp
\Bigl(C\int_0^t\|v^3(t')\|_{\dH^{\frac 1 2   +\frac 2 p
}}^p\,dt'\Bigr)\b  iggr). \eeno }
\end{prop}

\begin{proof}
The  important point  is the proof of  the following estimate: for any~$t$ in~$[0,T^\star[$, we have
 \beq
 \label{b.28}
 \begin{split}
\bigl\|\om_{\f r
2}(t)\bigr\|_{L^2}^{2\left(1+2p\al(r)\right)}+&\bigl\|\na\om_{\frac
r 2}\bigr\|_{L^2_t(L^2)}^{2\left(1+2p\al(r)\right)}\\
& \leq C \|\Om_0\|_{L^{r}}^{r\left(1+2p\al(r)\right)}
\exp\biggl(C\exp \Bigl(C\int_0^t\|v^3(t')\|_{\dH^{\frac 1 2   +\frac
2 p }}^p\,dt'\Bigr)\biggr).\end{split} \eeq In order to do it, let
us  introduce the notation \beq\label{b.289q} \ds e(T)\eqdefa C\exp
\Bigl(C\int_0^T \|v^3(t)\|_{\dH^{\frac 1 2   +\frac 2 p }}^p
dt\Bigr). \eeq where the constant~$C$ may change from line to line.
As~$(a+b)^{\frac r 2} \sim a^{\frac r 2}+b^{\frac r 2}$,
Proposition\refer{estimadivhaniso} implies that \beq
 \label{SophisticatedGronwalldemoeq2}
\begin{split}
\Bigl(\int_0^t \|\p^2_3v^3  (t')
\|^{2}_{\cH^{\theta,r}}\,dt'\Bigr)^{\frac {r}{2}}e(T) \lesssim e(T)\bigl(\|\Om_0\|_{L^{r}}^{r}+V_1(t)+V_2(t)\bigr)\with\qquad\qquad\\
V_1(t) \eqdefa
 \biggl(\int_0^t\|v^3(t')\|_{\dH^{\frac 1 2   +\frac 2 p }}
\bigl\| \, \om_{\frac r 2}(t')\bigr\|_{L^2}^{2\left(2\al(r)+
\frac 1 p \right)} \bigl\|\na\om_{\frac r
2}(t')\bigr\|_{L^2}^{\frac 2 {p'}}\,dt' \biggr)^{\frac r 2}
\andf \\
V_2(t) \eqdefa \biggl(\int_0^t\|v^3(t')\|_{\dH^{\frac 1 2   +\frac 2 p }}^2 \bigl\|
\, \om_{\frac r 2}(t')\bigr\|_{L^2}^{4\left(\al(r)+\frac 1 p \right)}
\bigl\|\na\om_{\frac r
2}(t')\bigr\|_{L^2}^{2\left(1-\frac 2 p \right)}\,dt' \biggr)^{\frac r
2}.\qquad
\end{split}
\eeq Let us estimate the two terms~$V_j(t), j=1,2$. Applying
H\"older inequality gives
 \beno
  V_1(t)  \leq
\biggl(\int_0^t\|v^3(t')\|_{\dH^{\frac 1 2   +\frac 2 p }}^p \bigl\| \om_{\frac r
2}(t')\bigr\|_{L^2}^{2\left(1+2p\al(r)\right)}\,dt'\biggr)^{\frac r
2 \times\frac1 p} \biggl(\int_0^t\bigl\|\na\om_{\frac r
2}(t')\bigr\|_{L^2}^{2}\,dt' \biggr)^{\frac r
2\left(1-\frac1p\right)}.
 \eeno
 As we have
$$
1-\frac r 2\Bigl(1-\frac1p\Bigr) =
r\Bigl(\al(r)+\f1{2p}\Bigr)=\frac{r\left(1+2p\al(r)\right)}{2p}\,\virgp
$$
convexity inequality implies that, for any~$t$ in~$[0,T]$, \beq
\label{SophisticatedGronwalldemoeq3} \begin{split} e(T) V_1(t) \leq
\frac{r-1}{3r^2} &\int_0^t\bigl\|\na\om_{\frac r
2}(t')\bigr\|_{L^2}^{2}\,dt'\\
&{} + e(T) \biggl(\int_0^t\|v^3(t')\|_{\dH^{\frac 1 2   +\frac 2 p }}^p \bigl\|
\om_{\frac r
2}(t')\bigr\|_{L^2}^{2\left(1+2p\al(r)\right)}\,dt'\biggr)^{\frac
1{1+2p\al(r)}}. \end{split} \eeq Now let us estimate the
term~$V_2(t)$. Applying H\"older inequality yields \beno V_2(t) \leq
\biggl(\int_0^t\|v^3(t')\|_{\dH^{\frac 1 2   +\frac 2 p }}^p \bigl\| \om_{\frac r
2}(t')\bigr\|_{L^2}^{2\left(1+p\al(r)\right)}\,dt'\biggr)^{\frac r
2\times\frac2 p} \biggl(\int_0^t\bigl\|\na\om_{\frac r
2}(t')\bigr\|_{L^2}^{2}\,dt' \biggr)^{\frac r
2\left(1-\frac2p\right)}.
 \eeno
 As we have
$$
1-\frac r 2\left(1-\frac2p\right) = r\left(\al(r)+\f1p\right)=\frac
{r\left(1+p\al(r)\right)}{p}\,\virgp
$$
convexity inequality implies that \beq
\label{SophisticatedGronwalldemoeq4} \begin{split}
 e(T) V_2(t)
 \leq \frac{r-1}{3r^2} &
\int_0^t\bigl\|\na\om_{\frac r 2}(t')\bigr\|_{L^2}^{2}\,dt'\\
& + e(T) \biggl(\int_0^t\|v^3(t')\|_{\dH^{\frac 1 2   +\frac 2 p }}^p \| \om_{\frac
r 2}(t')\|_{L^2}^{2\left(1+p\al(r)\right)}\,dt'\biggr)^{\frac
1{1+p\al(r)}}.\end{split} \eeq
 Let us notice that the power of~$\bigl\|  \om_{\frac r 2}\bigr\|_{L^2}$ here is not the
same as that in Inequality\refeq{SophisticatedGronwalldemoeq3}.
Applying H\"older inequality with
$$
q = \frac {1+2p\al(r)} {1+p\al(r)}
$$
and with the measure~$\|v^3(t')\|_{\d  H^{\frac 1 2   +\frac 2 p }}^p\,dt'$ gives
$$
\longformule{ \biggl(\int_0^t\|v^3(t')\|_{\dH^{\frac 1 2   +\frac 2 p }}^p \bigl\|
\om_{\frac r
2}(t')\bigr\|_{L^2}^{2\left(1+p\al(r)\right)}dt'\biggr)^{\frac
1{1+p\al(r)}} \leq
\biggl(\int_0^t\|v^3(t')\|_{\dH^{\frac 1 2   +\frac 2 p }}^p\,dt'\biggr)^{\bigl(1-\frac
1q\bigr)\times\frac1{1+p\al(r)}} } {{}\times
\biggl(\int_0^t\|v^3(t')\|_{\dH^{\frac 1 2   +\frac 2 p }}^p \bigl\| \om_{\frac r
2}(t')\bigr\|_{L^2}^{2\left(1+2p\al(r)\right)}\,dt'\biggr)^{\frac
1{1+2p\al(r)}}. }
$$
By definition of~$e(T)$, we have
$$
\biggl(\int_0^t\|v^3(t')\|_{\dH^{\frac 1 2   +\frac 2 p }}^p\,dt'\biggr)^{\bigl(1-\frac
1q\bigr)\times\frac1{1+p\al(r)}} e(T)\leq e(T).
$$
Thus we deduce from\refeq{SophisticatedGronwalldemoeq4} that
$$
e(T) V_2(t) \leq \frac{r-1}{3r^2} \int_0^t\bigl\|\na\om_{\frac r
2}(t')\bigr\|_{L^2}^{2}dt' + e(T)
\biggl(\int_0^t\|v^3 (t')\|_{\dH^{\frac 1 2   +\frac 2 p }}^p \bigl\|\om_  {\frac r
2}(t')\bigr\|_{L^2}^{2\left(1+2p\al(r)\right)}\,dt'\biggr)^{\frac
1{1+2p\al(r)}}.
$$
Inserting this inequality and\refeq{SophisticatedGronwalldemoeq3}
in\refeq{SophisticatedGronwalldemoeq2} gives, for any~$t$
in~$[0,T]$,
$$
\longformule{ \Bigl(\int_0^t \|\p^2_3v^3(t')
\|^{2}_{\cH^{\theta,r}}\,dt'\Bigr)^{\frac {r}{2}}e(T) \leq
\frac{2(r-1)}{3r^2}  \int_0^t\bigl\|\na\om_{\frac r
2}(t')\bigr\|_{L^2}^{2}\,dt'+e(T)\|\Om_0\|_{L^{r}}^{r} } {{} + e(T)
\biggl(\int_0^t\|v^3(t')\|_{\dH^{\frac 1 2   +\frac 2 p }}^p \bigl\| \om_{\frac r
2}(t')\bigr\|_{L^2}^{2\left(1+2p\al(r)\right)}\,dt'\biggr)^{\frac
1{1+2p\al(r)}}. }
$$
Hence thanks to Proposition\refer{inegfondvroticity2D3D}, we deduce
that
$$
\longformule{ \frac 1r  \|  \om_{\frac r2}(t)\|_{L^2}^{2}
 +\frac{r-1}{3r^2} \int_0^t\|\nabla \om_{\frac r2}(t')\|_{L^2}^2\,dt' \leq  \|\Om_0\|_{L^{r}}^{r}e(T)
} {{} + e(T) \biggl(\int_0^t  \|v^3(t')\|_{\dH^{\frac 1 2   +\frac 2 p }}^p
\|\om_{\frac
r2}(t')\|_{L^2}^{2\left(1+2p\al(r)\right)}\,dt'\biggr)^{\frac
1{1+2p\al(r)}}. }
$$
Taking the power~$\ds 1+2p\al(r)$ of this inequality  and using that
$$
(a+b)^{ 1+2p\al(r)} \sim a^{1+2p\al(r)}+b^{1+2p\al(r)},
$$
 we obtain for any~$t$ in~$[0,T]$,
$$
\longformule{
 \bigl\|  \om_{\frac r2}(t)\bigr\|^{2\left(1+2p\al(r)\right)}_{L^2}
 +\biggl(\int_0^t\bigl\|\nabla \om_{\frac r2}(t')\bigr\|_{L^2}^2\,dt'\biggr)^{1+2p\al(r)}
  \leq \| \Om_0\|_{L^{r}}^{r\left(1+2p\al(r)\right)}e(T)
 }
{{} + e(T)\int_0^t\|v^3(t')\|_{\dH^{\frac 1 2   +\frac 2 p }}^p \bigl\| \om_{\frac
r2}(t')\bigr\|_{L^2}^{2\left(1+2p\al(r)\right)}\,dt'. }
$$
Then Gronwall lemma leads to Inequality \eqref{b.28}.
On the other hand, it follows from
Proposition\refer{estimadivhaniso}  that, for any $t<T^\ast,$ \beno
&& \|\p_3v^3(t)\|_{\cH^{\theta,r} }^2
 +\int_0^t\|\na\p_3v^3(t')\|_{\cH^{\theta,r}}^2\,d  t'\\
&&\qquad\qquad \leq
 e(t)\biggl(\|\Om_0\|_{L^{r}}^2+\|v^3\|_{L^p_t(\dH^{\frac 1 2   +\frac 2 p })}\bigl\|\om_{\frac r2}\bigr\|_{L^\infty_t(L^2)}^{2\left(2\al(r)+\frac 1 p \right)}
 \bigl\|\na\om_{\frac r2}\bigr\|_{L^2_t(L^2)}^{\frac 2 {p'}}\\
 &&\qquad\qquad\qquad\qquad\qquad\qquad\qquad{}+\|v^3\|_{L^p_t(\dH^{\frac 1 2   +\frac 2 p })}^2\bigl\|\om_{\frac r2}\bigr\|_{L^\infty_t(L^2)}^{4\left(\al(r)+\frac 1 p \right))}
 \bigl\|\na\om_{\frac r2}\bigr\|_{L^2_t(L^2)}^{2\left(1-\frac 2 p \right)}\biggr).
 \eeno
Inserting the Estimate \eqref{b.28} in the above inequality
concludes the proof of Proposition \ref{SophisticatedGronwall}.
\end{proof}

Thus, if we assume that
\beq
\label{k.1vertical}
\int_0^{T^\star}\|v^3(t)\|_{\dH^{\frac 1 2   +\frac 2 p }}^p\,dt<\infty,
\eeq
we know that  the quantities
\beq
\label{eqconclud1}
\|\om\|_{L^\infty([0,T^\star[; L^{r})},\quad
 \int_0^{T^\star}\|\nabla \om_{\f r 2}(t)\|_{L^2}^2\,dt\,,\andf  \int_0^{T^\star} \|\p^2_3v^3(t)
\|^{2}_{\cH^{\theta,r}}\,dt \eeq are finite. We want to prove that
it prevents this solution from blowing up. Let us recall the
following theorem of anisotropic condition for blow up.
\begin{thm}[Theorem 2.1 of \cite{CZ5}]
\label{blowupBesovendpoint}
{\sl
 Let $v$ be a solution of~$(NS)$  in the
space~$C([0,T^\star[;\dH^{\frac 1 2   })\cap L^2_{\rm
loc}([0,T^\star[;H^{3/2})$. If~$T^\star$ is the maximal time of
existence and $T^\ast<\infty,$ then for any~$(p_{k,\ell})$
in~$]1,\infty[^9$, one has
$$
\sum_{1\leq k,\ell\leq3} \int_0^{T^\star} \|  \partial_\ell
v^k(t)\|^{p_{k,\ell}}_{\cB_{p_{k,\ell}}} dt=\infty,
$$
where $\cB_p\eqdefa\dB^{-2+\frac 2 p }_{\infty,\infty}.$}
\end{thm}

Now let us present the proof of Theorem \refer{thmain}.

\begin{proof}[Proof of  Theorem \refer{thmain}]
We first deduce  from Lemma \ref{lemBern} that
 \beno
\max_{1\leq\ell\leq 3} \|\p_\ell v^3\|_{\cB_p}\lesssim
\sup_{j\in\Z}2^{j\left(-1+\frac 2 p \right)}\|\D_jv^3\|_{L^\infty}\lesssim
\sup_{j\in\Z}2^{j\left(\frac 1 2   +\frac 2 p \right)}\|\D_jv^3\|_{L^2}\lesssim
\|v^3\|_{\dH^{\frac 1 2   +\frac 2 p }},
\eeno
which ensures that
\beq
 \label{b.29}
 \max_{1\leq\ell\leq3}\int_0^{T^\star}\|\p_\ell v^3(t)\|_{\cB_p}^p\,dt\lesssim
\int_0^{T^\star}\|v^3(t)\|_{\dH^{\frac 1 2   +\frac 2 p }}^p\,dt<\infty.
 \eeq
As we have
$$
\|\partial_{\rm h}^2\D_{\rm
h}^{-1}\p_3 v^3(t)\|_{\cB_p} \lesssim
\|\partial_{\rm h}^2\D_{\rm h}^{-1}\p_3 v^3(t)\|_{\dot H^{-\frac 12+\frac 2p}},
$$
so that for~$v^{\rm h}_{\rm div} =-\nablah\D_{\rm h}^{-1}
\partial_3v^3,$ there holds
 \beq
 \label{b.30}
  \int_0^{T^  \ast}\|\na_{\rm h} v^{\rm h}_{\rm div}(t)\|_{\cB_p}^p\,dt\lesssim
\int_0^{T^\ast}\|v^3(t)\|_{\dH^{\frac 1 2   +\frac 2 p }}^p\,dt<\infty.
 \eeq
 The other components of the matrix~$\nabla v$ can been estimated with norm which are not of scaling~ $0$,
  namely  norms related to ~$L^r$ regularity of the horizontal
  vorticity~$\om$. To proceed further, we get, by
using Lemma\refer{lemBern}, that \beno
\|\D_ja\|_{L^\infty}&\lesssim& \sum_{\substack{k\leq j+1\\\ell\leq
j+1}} 2^k2^{\f{\ell}2}\|\D_k^{\rm h}\D_{\ell}^{\rm
v}a\|_{L^2}\\
&\lesssim& \|a\|_{\dot
H^{1-3\alpha(r)+\theta,-\theta}}\sum_{\substack{k\leq j+1\\\ell\leq
j+1}}2^{k(3\al(r)-\th)}2^{\ell\bigl(\f12+\th\bigr)}\\
&\lesssim& 2^{j\bigl(\f12+3\al(r)\bigr)}\|a\|_{\dot
H^{1-3\alpha(r)+\theta,-\theta}}, \eeno
 because~$\ds
 -\frac 12-3\alpha(r) = -2+\frac 3 {r'}\,\virgp$ this leads to
 \beq
\label{eqconclud2}
\|a\|_{\cB_{\frac {2r'} 3}} \lesssim \|a\|_{\dot H^{1-3\alpha(r)+\theta,-\theta}}.
 \eeq
 Let us define~$\ds q(r)\eqdefa \frac {2r'} 3$. As~$r$ belongs to~$\ds ]3/2, 2[$,~$q(r)$ is in~$\ds ]4/3, 2[$
 and thus is less than ~$2$. Observing that
 $$
\|\partial_3v^{\rm h}_{\rm div}\|_{\dot H^{1-3\alpha(r)+\theta,-\theta}} =  \|\nabla_{\rm h} \Delta_{\rm h} ^{-1} \partial_3^2 v^3\|_{\dot H^{1-3\alpha(r)+\theta,-\theta}} \lesssim
\|\partial_3^2 v^3\|_{\cH^{\theta,r} },
 $$
 then applying Inequality\refeq{eqconclud2} and H\"older inequality, we deduce that
  \beq
 \label{b.31}
  \int_0^{T^\ast}\|\partial_3 v^{\rm h}_{\rm div}(t)\|_{\cB_{q(r)}}^{q(r)}\,dt\lesssim
{T^{\star}}^{\left( 1-\frac {q(r)} 2\right)} \Bigl(\int_0^{T^\ast}\|\partial_3^2v^3(t)\|_{\cH^{\theta,r}}^2\,dt
\Bigr)^{\frac {q(r)}2}<\infty\,.
 \eeq
 Let us admit for a while that
\beq
  \label{eqconclud3}
 \|\nabla v^{\rm h }_{\rm curl}(t)   \|_{\cB_{q(r)} }\lesssim \|\nabla \om(t)\|_{L^r}.
 \eeq
 Lemma\refer{BiotSavartomega} implies  that
 $$
  \|\nabla \om(t) \|_{L^{r}} \lesssim
  \|\om_{\frac{r}2}\|_{L^\infty([0,T^\star;L^2)}^{\frac 2 r-1} \bigl\|\nabla \om_{\frac{r}2}(t)\bigr\|_{L^2}\,.
 $$
 Then H\"older inequality implies that
 $$
   \int_0^{T^\ast}\|\nabla v^{\rm h}_{\rm curl}(t)\|_{\cB_{q(r)}}^{q(r)}\,dt\lesssim
{T^{\star}}^{\left( 1-\frac {q(r)} 2\right)}
 \|\om_{\frac{r}2}\|_{L^\infty([0,T^\star[;L^2)}^{\frac 2 r-1}
 \Bigl(\int_0^{T^\ast}\|\nabla \om_{\frac{r}2}(t)\bigr\|_{L^2}^2\,dt \Bigr)^{\frac {q(r)}2}<\infty\,.
 $$
Together with Inequalities\refeq{b.29},\refeq{b.30} and\refeq{b.31},
this concludes the proof of Theorem\refer{thmain} provided we prove
the Estimate\refeq{eqconclud3}.

\medbreak
Let us start with the term~$\nabla_{\rm h} v^{\rm h}_{\rm curl}$.  Dual Sobolev embedding implies that
$$
\|\om\|_{\dot H^{1-3\alpha(r)}}
\lesssim \|\nabla \om\|_{\dot H^{-3\alpha(r)}}
\lesssim \|\nabla \om\|_{L^r}.
$$
As~$ \nabla_{\rm h} v^{\rm h}_{\rm curl}= \partial_{\rm
h}^2\Delta_{\rm h}^{-1} \om$, we get, by using Lemma \ref{isoaniso}
and \refeq{eqconclud2}, that \beq \label{eqconclud4}
 \|\nabla_{\rm h} v^{\rm h}_{\rm curl}\|_{\cB_{q(r)}} \lesssim
 \| \partial_{\rm h}^2\Delta_{\rm h}^{-1} \om\|_{\dot H^{1-3\alpha(r)}}\lesssim \|\nabla \om\|_{L^r}\,.
\eeq
 The term~$\partial_3 v^{\rm h}_{\rm curl}$ is treated as follows. Let us write that
 $$
 \D_j\partial_3 v^{\rm h}_{\rm curl}  = \sum_{\substack{k\leq j+1\\\ell\leq j+1}}
 \D_j\D_k^{\rm h}\D_\ell^{\rm v} \partial_3\nabla_{\rm h}^\perp\Delta_{\rm h}^{-1} \om\,.
 $$
 Using Lemma\refer{lemBern}, we can wri  te
\beno
 2^{-j\left ( \frac 3 r-1\right)}\|\D_j\partial_3 v^{\rm h}_{\rm curl} \|_{L^\infty}
 & \lesssim &   2^{-j\left ( \frac 3 r-1\right)}
\sum_{\substack{k\leq j+1\\\ell\leq j+1}}
\bigl\| \D_j\D_k^{\rm h}\D_\ell^{\rm v} \partial_3\nabla_{\rm h}^\perp\Delta_{\rm h}^{-1} \om\bigr\|_{L^\infty}\\
& \lesssim &
 \|\partial_3\om\|_{L^r}  2^{-j\left ( \frac 3 r-1\right)} \sum_{\substack{k\leq j+1\\\ell\leq j+1}} 2^{k\left (\frac 2 r-1\right)} 2^{\frac \ell r} \\
& \lesssim &  \|\partial_3\om\|_{L^r} \,. \eeno
 This concludes the proof of \refeq{eqconclud3} and hence also the proof of  Theorem\refer{thmain} . \end{proof}

\bigbreak \noindent {\bf Acknowledgments.} Part of this work was
done when J.-Y. Chemin was visiting Morningside Center of the
Academy of Mathematics and Systems Sciences, CAS. We appreciate the
hospitality and the financial support from MCM and National Center
for Mathematics and Interdisciplinary Sciences. P. Zhang is
partially supported by NSF of China under Grant  11371347, the
fellowship from Chinese Academy of Sciences and innovation grant
from National Center for Mathematics and Interdisciplinary Sciences.
Z. Zhang is partially supported by NSF of China under Grant
11371037 and 11425103, Program for New Century Excellent Talents in University
and Fok Ying Tung Education Foundation.


\medskip


\begin{thebibliography}{50}


\bibitem{BCD} H. Bahouri, J.~Y. Chemin and R. Danchin, {\it Fourier analysis and
nonlinear partial differential equations}, Grundlehren der
mathematischen Wissenschaften 343, Springer-Verlag Berlin
Heidelberg, 2011.



\bibitem{Bo81} J.-M. Bony, Calcul symbolique et propagation des
   singularit\'es pour les \'equations aux d\'eriv\'ees partielles non
   lin\'eaires, {\it Annales de l'\'Ecole Normale Sup\'erieure,} {\bf
   14}, 1981, pages 209-246.






\bibitem{CDGG} J.-Y. Chemin, B. Desjardins, I. Gallagher and
 E. Grenier, Fluids with anisotropic viscosity, {\it Mod\'elisation
 Math\'ematique et Analyse Num\'erique}, {\bf 34}, 2000, pages 315-335.






\bibitem {CZ1}
J.-Y. Chemin and P.  Zhang, On the global wellposedness  to the 3-D incompressible anisotropic
 Navier-Stokes equations, {\it Communications in  Mathematical  Physics}, {\bf 272}, 2007, pages 529--566.

 \bibitem{CZ5}
J.-Y. Chemin and P.  Zhang, On the critical one component regularity
for 3-D
 Navier-Stokes system,  arXiv:1310.6442[math.AP], accepted by  {\it Annales de l'\'Ecole Normale Sup\'erieure} on 2014.

 \bibitem{ISS}  L. Escauriaza, G. Seregin and V. \u{S}ver\'{a}k,
$L^{3,\infty}$  -solutions of Navier-Stokes equations and backward
uniqueness, (Russian) {\it Uspekhi Mat. Nauk}, {\bf 58}, 2003, no.
2(350), pages 3-44; translation in {\it Russian Math. Surveys}, {\bf
58} , 2003, pages  211-250.




\bibitem{fujitakato}
H. Fujita and T. Kato, On the Navier-Stokes initial value problem I,
{\em Archive for Rational Mechanic Analysis}, {\bf 16}, 1964, pages
269--315.













\bibitem{Pa02} M. Paicu, \'Equation anisotrope
de Navier-Stokes dans des espaces  critiques,  {\it Revista
Matem\'atica Ibero\-americana,} {\bf 21}, 2005, pages   179--235.











\end{thebibliography}
\end{document}